\documentclass[11pt,letterpaper]{amsart}

\usepackage{amsmath,amssymb,amsfonts,amsthm,mathtools}
\usepackage[margin=1.15in]{geometry}
\usepackage[colorlinks=true,citecolor=blue,linkcolor=blue,urlcolor=blue]{hyperref}
\usepackage{enumitem}

\numberwithin{equation}{section}
\allowdisplaybreaks

\newtheorem{theorem}{Theorem}[section]
\newtheorem{lemma}[theorem]{Lemma}
\newtheorem{corollary}[theorem]{Corollary}
\newtheorem{remark}[theorem]{Remark}
\theoremstyle{definition}

\newtheorem*{remark*}{Remark}

\newcommand{\N}{\mathbb N}
\newcommand{\Z}{\mathbb Z}
\newcommand{\R}{\mathbb R}
\newcommand{\F}{\mathbb F}
\newcommand{\e}{\mathrm e}
\newcommand{\Kl}{\operatorname{Kl}}
\newcommand{\disc}{\operatorname{Disc}}
\newcommand{\Res}{\operatorname{Res}}
\newcommand{\lc}{\operatorname{lc}}
\newcommand{\rad}{\operatorname{rad}}
\newcommand{\Om}{\Omega}

\newcommand{\bnu}{\boldsymbol\nu}
\newcommand{\sn}{\mathfrak s}

\title[Mixed moments of arithmetic functions and Kloosterman sums]
{Mixed moments of arithmetic functions and Kloosterman sums}

\author{Yujiao Jiang}

\address{Yujiao Jiang\\
    School of Mathematics and Statistics
    \\
    Shandong University
    \\
    Weihai
    \\
    Shandong 264209
    \\
    China}
\email{yujiaoj@sdu.edu.cn}

\author{Yuk-Kam Lau}

\address{Yuk-Kam Lau\\
    Department of Mathematics
    \\ 
    The University of Hong Kong
    \\ 
    Pokfulam Road
    \\ 
    Hong Kong\\
    China}
\email{yklau@maths.hku.hk}
\date{\today}

\begin{document}
\begin{abstract}
We extend the Nair--Tenenbaum estimates for short sums of nonnegative arithmetic functions at polynomial values to mixed moments involving Kloosterman sums. Combining these estimates with lower bounds obtained by refining the method of Fouvry and Michel, we prove that, for every fixed
nonzero integer \(a\) and fixed \(z,\nu>0\),
\[
\sum_{n\leq x}z^{\omega(n)}
|\operatorname{Kl}(a;n)|^\nu
\asymp_{a,z,\nu}
x(\log x)^{z\mathfrak{s}(\nu)-1},
\]
where \(\omega(n)\) counts distinct prime divisors and \(\mathfrak{s}(\nu)\) is the \(\nu\)-th absolute moment of the Sato--Tate measure. We also obtain upper bounds for mixed absolute moments of Hecke eigenvalues and Kloosterman sums, determine the frequency of large Kloosterman values on a logarithmic scale, and establish the order of magnitude predicted by Li and Sarnak for the diagonal second moment of classical Kloosterman sums.
We obtain matching bounds for the smoothed spectral counting variance on the modular surface in a logarithmic range and the lower bound \(\int_T^{2T}S(t)^2\,dt\gg T^2/\log T\), where \(S(t)\) is the remainder in Weyl's law.
\end{abstract}

\maketitle

\section{Introduction}\label{section1}

For integers \(a,b\) and \(m\geq1\), define
\begin{equation}\label{eq:S-def}
 S(a,b;m)
 :=\sideset{}{^*}\sum_{u\bmod m}
 \e\left(\frac{au+b\bar u}{m}\right),
\end{equation}
where $\sideset{}{^*}\sum$ denotes summation over reduced residue classes, $\mathrm{e}(z)=\exp(2\pi i z)$, and $\bar u$ is the multiplicative inverse of $u$ modulo $m$. The substitution \(u\mapsto \bar u\) gives \(S(a,b;m)=S(b,a;m)\).  We use the normalization
\begin{equation*}
 \Kl(a;m):=m^{-1/2}S(a,1;m).
\end{equation*}
The Chinese remainder theorem gives the
twisted multiplicativity formula
\begin{equation}\label{eq:Kl-mult}
 \Kl(a;mn)=\Kl(a\bar n^2;m)\Kl(a\bar m^2;n)
\end{equation}
for $(m,n)=1$, see \cite[(1.58)--(1.59)]{IK-2004}. Write \(\omega(n)\) for the
number of distinct prime factors of \(n\), and \(\Omega(n)\) for the
number counted with multiplicity. The Weil--Estermann bound gives
\begin{equation}\label{eq:Kl-Weil}
 |\Kl(a;m)|\leq \kappa(v_2(m))\cdot 2^{\omega(m)}\le 2^{3/2}2^{\omega(m)},
\end{equation}
where $v_2(m)$ is the $2$-adic valuation of $m$, and $\kappa(x) =1,\sqrt{2}, 2,2^{3/2}$ according as $x$ lies in $[0,4]$, $(4,5]$, $(5,6]$ or $(6,\infty)$; see \cite[(1.1)]{ASS-2024}.

Let \(k\geq1\), and let \(Q_1,\ldots,Q_k\in\Z[X]\) be pairwise coprime
irreducible polynomials, and put
\[
 Q=Q_1\cdots Q_k,\qquad g=\deg Q,\qquad D=\disc(Q)
\]
where $\deg$ and $\disc$ denote the degree and discriminant respectively.
Then \(Q\) is squarefree over \(\mathbb Q\), so \(D\ne0\).
For \(a\in\Z\setminus\{0\}\),
\(\bnu=(\nu_1,\ldots,\nu_k)\in\R_{\geq0}^k\), and a nonnegative
arithmetic function \(F\) satisfying the growth conditions below, we
consider
\[
 \sum_{\substack{x<n\leq x+y\\Q(n)\ne0}}
 F(|Q_1(n)|,\ldots,|Q_k(n)|)
 \prod_{j=1}^k|\Kl(a;|Q_j(n)|)|^{\nu_j}.
\]
Factors with zero exponent are omitted.

For fixed polynomials \(Q_1,\ldots,Q_k\) satisfying our hypotheses and \(F\equiv1\), this case is of particular interest and has appeared in the work of several authors. For instance, when \(\nu_1=\cdots=\nu_k=1\), our result gives the upper bound \(y(\log x)^{-k(1-8/(3\pi))}\), up to a constant depending on the fixed parameters. Furthermore, when \(k=1\) and \(Q_1(x)=x\), our result generalizes and improves the upper bound result in Fouvry and Michel \cite{FM-2003}. The new upper bound in this case is tight. 
Thanks to the work of Fouvry and Michel, we can complement it with a nice lower bound by modifying their method in \cite{FM-2003}. In summary, we obtain that for every
fixed \(a\ne0\), \(z>0\), and \(\nu>0\), 
\[
 \sum_{n\leq x}z^{\omega(n)}|\Kl(a;n)|^\nu
 \asymp_{a,z,\nu}x(\log x)^{z\sn(\nu)-1},
\]
where \(\sn(\nu)\) is the Sato--Tate moment defined in \eqref{eq:sato-moment}.
The lower bound holds even on squarefree moduli coprime to \(2a\),
and is uniform for \(0<|a|\leq x^A\), with an explicit factor at
primes dividing \(2a\); see Theorem~\ref{thm:lower}. 
More applications will be given in Section~\ref{sec:applications}.


We now state the arithmetic hypotheses and motivation.
For \(k\geq1\), \(A,B\geq1\),
and \(\varepsilon>0\), let \(\mathcal M_k(A,B,\varepsilon)\) be the
class of functions \(F:\N^k\to\R_{\geq0}\) with
\(F(1,\ldots,1)=1\) such that
\begin{equation}\label{eq:Mclass}
 F(a_1b_1,\ldots,a_kb_k)
 \leq
 \min\left(A^{\Omega(a_1\cdots a_k)},
 B(a_1\cdots a_k)^\varepsilon\right)
 F(b_1,\ldots,b_k)
\end{equation}
whenever \((a_1\cdots a_k,b_1\cdots b_k)=1\).
For \(Q(X)=\sum_{j=0}^g c_jX^j\in\Z[X]\), write
\begin{equation}\label{eq:Q-rho}
 \|Q\|:=\sum_{j=0}^g|c_j|,
 \qquad
 \rho_Q(m):=\#\{r\in\Z/m\Z:Q(r)\equiv0\,(\text{mod}\,m)\}.
\end{equation}

The class $\mathcal{M}_k$ has been studied previously. In 1980, Shiu \cite{Shiu-1980} proved a Brun--Titchmarsh theorem for
multiplicative functions in short arithmetic progressions.
Nair \cite{Nair-1992} extended the method to polynomial values, and
Nair and Tenenbaum \cite{NT-1998} developed a multivariable version, thereby broadening its range of applications. Henriot \cite{Henriot-2012,Henriot-2014-erratum} refined the results of \cite{NT-1998} by  making the dependence on
the discriminant explicit, using Stewart's estimate
\cite{Stewart-1991} for polynomial congruences. Let us recall the following special case, in the formulation of
\cite[Theorem~1]{Henriot-2012}.

\begin{theorem}[Henriot-Nair-Tenenbaum]\label{thm-nt}
Let \(k\geq1\), and let \(Q_1,\ldots,Q_k\in\Z[X]\) be pairwise coprime
irreducible polynomials. Put \(Q=Q_1\cdots Q_k\), \(g=\deg Q\),
\(D=\disc(Q)\), and \(\rho=\rho_Q\), and suppose that \(Q\) has no
fixed prime divisor. Let \(0<\alpha,\delta<1\), \(A,B\geq1\), and
\(0<\varepsilon\leq\alpha\delta/(12g^2)\). If
\(F\in\mathcal M_k(A,B,\varepsilon)\), then, uniformly for
\(x\geq c_0\|Q\|^\delta\) and \(x^\alpha<y\leq x\),
\begin{equation*}\label{eq-nt}
\begin{split}
 &\sum_{\substack{x<n\leq x+y\\Q(n)\ne0}}
 F(|Q_1(n)|,\ldots,|Q_k(n)|)\\
 &\quad\ll y\prod_{p\leq x}\left(1-\frac{\rho(p)}p\right)
 \sum_{n_1\cdots n_k\leq x}
 F(n_1,\ldots,n_k)
 \frac{\rho_{Q_1}(n_1)\cdots\rho_{Q_k}(n_k)}{n_1\cdots n_k}.
\end{split}
\end{equation*}
The constant \(c_0\) depends only on \(g,\alpha,\delta,A,B\), whereas
the implied constant may additionally depend on \(D\).
\end{theorem}

It should be remarked that Theorem~\ref{thm-nt} {\it cannot} be applied to our problem, because Kloosterman sums are twisted multiplicative but not multiplicative, see \eqref{eq:Kl-mult}. The dependence on the complementary modulus precludes the Kloosterman weights from satisfying the identity \eqref{eq:Mclass}. 


Section 2 lists notation and a crucial condition. Sections 3 and 4 state
the main results and applications. Sections 5 and 6 establish the local
estimates and sieve lemmas. Section 7 proves Theorem~\ref{thm:main}, and
Section~\ref{sec:upper-corollaries} proves its corollaries.
Section~\ref{sec:lower-proof} proves Theorem~\ref{thm:lower}.
The proofs of the applications are given in
Section~\ref{sec:application-proofs}.

\section{Notation and weighted congruence densities}

The parameters \(k,g,A,B,\alpha,\delta,\varepsilon\), and
\(\boldsymbol\nu\) are fixed, subject to the hypotheses of the results
below. The nonzero integer \(a\) and the polynomial tuple
\(\mathbf Q=(Q_1,\ldots,Q_k)\) may vary. Unless stated otherwise, the
implied constants are independent of \(a\) and the polynomial
coefficients, and hence have no further dependence on \(D\).

For a nonzero polynomial \(R\in\Z[X]\), write \(\lc(R)\) for its
leading coefficient, and define \(\|R\|\) and \(\rho_R(m)\)  as in \eqref{eq:Q-rho}. We write \(d\parallel n\) if \(d\mid n\) and
\((d,n/d)=1\), and put \(\rad(d)=\prod_{p\mid d}p\).
Let \(P^+(n)\) and \(P^-(n)\) denote the largest and smallest prime
factors of \(n\), with \(P^+(1)=1\) and \(P^-(1)=+\infty\).
We use the usual valuation \(v_p(n)\), with \(v_p(0)=+\infty\).
Products of vectors are taken componentwise, i.e. Hadamard products, and
\(\mathbf1_k=(1,\ldots,1)\). The letter \(p, p_1,p_2,\ldots\) always denote  primes.
For \(\boldsymbol\ell\in\Z_{\geq0}^k\), write
\[
 |\boldsymbol\ell|=\ell_1+\cdots+\ell_k,
 \qquad p^{\boldsymbol\ell}=(p^{\ell_1},\ldots,p^{\ell_k}).
\]
We also write \(|\bnu|=\nu_1+\cdots+\nu_k\).

For \(\mathbf d=(d_1,\ldots,d_k)\in\N^k\) and \(r\in\Z\), let
\(\mathcal P_{\mathbf Q}(\mathbf d,r)\) denote the conditions
\[
 v_p(Q_j(r))=v_p(d_j)
 \qquad(p\mid d_1\cdots d_k,\ 1\leq j\leq k).
\]
Equivalently, the conditions are 
\(d_j\parallel Q_j(r)\) and \(p\nmid Q_j(r)\) at  \(p\mid d_1\cdots d_k\) with \(p\nmid d_j\),  for all $1\le j\le k$. These are the two
conditions in \cite[(0.1)]{Henriot-2014-erratum}.

For fixed \(\mathbf d\), we set
\[
 L(\mathbf d):=[d_1^2,\ldots,d_k^2].
\]
We note a key observation.
\begin{align}\label{key}
\mbox{The condition \(\mathcal P_{\mathbf Q}(\mathbf d,r)\) depends only on
\(r \bmod L(\mathbf d)\).}  \tag{K-O}   
\end{align}
Indeed, if \(r'\equiv r \pmod{L(\mathbf d)}\), then
\[
Q_j(r')\equiv Q_j(r)\pmod{d_j^2},
\]
and hence  \(\mathcal P_{\mathbf Q}(\mathbf d,r')\) holds if and only if  \(\mathcal P_{\mathbf Q}(\mathbf d,r)\)  holds. 
Whenever \(\mathcal P_{\mathbf Q}(\mathbf d,r)\) holds and \(d_j>1\), the quotient
\(Q_j(r)/d_j\) is a unit modulo \(d_j\), and
\begin{align}\label{p-property}
\frac{Q_j(r')}{d_j}\equiv \frac{Q_j(r)}{d_j}\pmod{d_j}.
\end{align}
Hence its inverse modulo \(d_j\) is independent of the chosen representative
\(r \bmod L(\mathbf d)\).  

We can thus partition the set of all \(r\in\mathbb Z\) satisfying
\(\mathcal P_{\mathbf Q}(\mathbf d,r)\) into admissible residue classes modulo
\(L(\mathbf d)\), which   justifies the definition
\begin{align}\label{def-H}
 H_{a,\boldsymbol\nu;\mathbf Q}(\mathbf d)
 :=\frac1{L(\mathbf d)}
 \sum_{\substack{r\bmod L(\mathbf d)\\
                 \mathcal P_{\mathbf Q}(\mathbf d,r)}}
 \prod_{j=1}^k
 \left|\Kl\left(a\,\overline{Q_j(r)/d_j}^{2};d_j\right)\right|^{\nu_j}.
\end{align}

Here and throughout, when \(d_j=1\), the corresponding Kloosterman factor is understood to be \(1\). We shall write \(H_{a,\boldsymbol\nu}\) when the polynomial tuple is clear from the context.

\section{Main results}

Theorem~\ref{thm:main} gives an upper bound in terms of the weighted
congruence density \(H_{a,\boldsymbol\nu}\). For practical applications, we separate the primes dividing the discriminant and give an explicit Euler-product bound; see Corollary~\ref{cor:discriminant}. The contribution from integers coprime to D admits an Euler-product bound, as shown in Corollary~\ref{cor:Henriot-cor2}. For fixed polynomials, these bounds yield powers of \(\log x\) determined by the local Kloosterman moments. We demonstrate such an upper-bound result in Corollary~\ref{cor:normalized}.
Theorem~\ref{thm:lower} gives a complementary lower bound when the modulus ranges over the positive integers, showing that the upper bound is sharp in this case.

For \(F\in\mathcal M_k(A,B,\varepsilon)\), define
\begin{equation}\label{eq:standard-majorant}
 G_F(\mathbf a) :=
 \sup_{\substack{\mathbf b\in\N^k\\
                  (a_1\cdots a_k,b_1\cdots b_k)=1\\
                  F(\mathbf b)>0}}
 \frac{F(a_1b_1,\ldots,a_kb_k)}{F(\mathbf b)}.
\end{equation}
For \((a_1\cdots a_k,b_1\cdots b_k)=1\), we have,  including when \(F(\mathbf b)=0\), 
\begin{align}\label{submultiplicativity}
F(\mathbf a\mathbf b)\le G_F(\mathbf a)F(\mathbf b)
\quad \mbox{ and } \quad 
G_F(\mathbf a\mathbf b)\leq G_F(\mathbf a)G_F(\mathbf b);
\end{align}
 by \eqref{eq:Mclass},
\begin{equation}\label{eq:standard-majorant-bound}
 G_F(\mathbf a)
 \leq
 \min \left(A^{\Om(a_1\cdots a_k)},
             B(a_1\cdots a_k)^\varepsilon\right).
\end{equation}
Taking \(\mathbf b=\mathbf1_k\) shows that \(G_F\geq F\).
Moreover, \(G_F\) is submultiplicative on coprime products, that is, \(G_F(\mathbf1_k) =1\) and 
\(G_F(\mathbf a\mathbf b)\leq G_F(\mathbf a)G_F(\mathbf b)\)
if \((a_1\cdots a_k,b_1\cdots b_k)=1\).

\begin{theorem}\label{thm:main}
Assume \(k\geq1\), \(a\in\Z\setminus\{0\}\), and
\(\bnu\in\mathbb R_{\geq0}^k\).  Let
\(Q_1,\ldots,Q_k\in\Z[X]\) be pairwise coprime irreducible
polynomials, put
\[
 Q=Q_1\cdots Q_k,\qquad g=\deg Q,\qquad
 D=\disc(Q),\qquad \rho=\rho_Q,
\]
and assume that \(Q\) is primitive.  Assume
\(0<\alpha,\delta<1\), \(A,B\geq1\), and let 
\begin{equation*}\label{eq:epsilon-range}
 0<\varepsilon
 \leq\frac{\alpha}{100g(g+1/\delta)}
\end{equation*}
and \(\mathcal{M}_k(A,B,\varepsilon)\) be defined as in Section~\ref{section1}.

If  \(F\in\mathcal M_k(A,B,\varepsilon)\), then, uniformly for
\begin{equation}\label{eq:xyrange}
 x\geq c_0\|Q\|^\delta,\qquad x^\alpha<y\leq x,    
\end{equation}
we have
\begin{equation}\label{eq:main-theorem}
\begin{split}
 &\sum_{\substack{x<n\leq x+y\\Q(n)\neq0}}
 F(|Q_1(n)|,\ldots,|Q_k(n)|)
 \prod_{j=1}^k|\Kl(a;|Q_j(n)|)|^{\nu_j}\\
 &\quad\ll
 y\prod_{g<p\leq x}\left(1-\frac{\rho(p)}p\right)
 \sum_{d_1\cdots d_k\leq x}
 F(d_1,\ldots,d_k)H_{a,\bnu}(d_1,\ldots,d_k).
\end{split}
\end{equation}
Here the constant \(c_0\) and the implied constant depend only on
\(\bnu,g,k,\alpha,\delta,A,B\).
\end{theorem}
\begin{remark}
    Since \(Q\) is
primitive, \(\rho(p)\leq g\) for every prime \(p\); hence all factors in the product are strictly positive.  A fixed prime divisor $p$ of \(Q\) can occur
only for \(p\leq g\), which is excluded from the
sifting product.
\end{remark}
Put
\begin{align}\label{eq:T-Delta-def}
 T_{a,\bnu,F}(\mathbf d)
 :=G_F(\mathbf d)H_{a,\bnu}(\mathbf d)
 \quad \mbox{ and } \quad
 \Delta_p(a,\bnu,F):=   1+\sum_{|\boldsymbol\ell|\geq1}
  T_{a,\bnu,F}(p^{\boldsymbol\ell}).    
\end{align}
For $0\neq S\in \Z$, define 
\begin{equation}\label{eq:Delta-D}
 \Delta_S(a,\bnu,F)
 :=
 \prod_{p\mid S}
 \left(
  1+\sum_{|\boldsymbol\ell|\geq1}
  T_{a,\bnu,F}(p^{\boldsymbol\ell})
 \right),
\end{equation}
where the product is interpreted as \(1\) when \(|S|=1\). Although the
notation suppresses \(Q_1,\ldots,Q_k\), the factor in
\eqref{eq:Delta-D} consists precisely of their weighted local congruence data
at the primes dividing \(D\).
 When $S=D$, it is the weighted analogue of Henriot's discriminant factor. For our purpose, we also consider $S=2D\lc(Q)$. 

\begin{corollary}
\label{cor:discriminant}
Under the assumption of Theorem \ref{thm:main},  for any $x,y$ satisfying \eqref{eq:xyrange},
\begin{equation}\label{eq:discriminant-separated}
\begin{split}
 &\sum_{\substack{x<n\leq x+y\\Q(n)\neq0}}
 F(|Q_1(n)|,\ldots,|Q_k(n)|)
 \prod_{j=1}^k|\Kl(a;|Q_j(n)|)|^{\nu_j}\\
 &\quad\ll
 \Delta_D(a,\bnu,F)\,y
 \prod_{g<p\leq x}\left(1-\frac{\rho(p)}p\right)
 \sum_{\substack{d_1\cdots d_k\leq x\\
                  (d_1\cdots d_k,D)=1\\
                  (d_i,d_j)=1\ (i\neq j)}}
 F(\mathbf d)H_{a,\bnu}(\mathbf d),
\end{split}
\end{equation}
where  $ \Delta_D(a,\bnu,F)$ (defined in \eqref{eq:Delta-D} with $S=\disc Q$) satisfies \begin{equation*}\label{eq:Delta-D-finite}
 \Delta_D(a,\bnu,F)
 \asymp
 \prod_{p\mid D}
 \left(
  1+\sum_{1\leq|\boldsymbol\ell|<8g}
  T_{a,\bnu,F}(p^{\boldsymbol\ell})
 \right).
\end{equation*}
In fact, 
\begin{equation*}\label{eq:Delta-D-bound}
 1\leq\Delta_D(a,\bnu,F)
 \leq\prod_{p\mid D}\left(1+\frac Cp\right)
 \leq\left(\frac{|D|}{\varphi(|D|)}\right)^C.
\end{equation*}
The implied constants in $\asymp$  and the constant \(C\ge 1\) depend on \(\bnu,g,k,A,B\) only, while the implied $\ll$-constant in \eqref{eq:discriminant-separated} depends only on \(\bnu,g,k,\alpha,\delta,A,B\). 
\end{corollary}

Let
\begin{equation*}\label{eq:bad-primes}
 \mathcal S
 :=\{p:p\leq g\}\cup\{p:p\mid2\lc(Q)D\}.
\end{equation*}

\begin{corollary}
\label{cor:Henriot-cor2}
For a prime \(p\), an
integer \(\ell\geq1\), and \(1\leq j\leq k\), write
\[
 \mathbf p_j^\ell:=(1,\ldots,1,p^\ell,1,\ldots,1),
\]
where \(p^\ell\) occupies the \(j\)-th coordinate in \(\mathbf p_j^\ell\). 
Under the assumptions of Theorem \ref{thm:main} holds, for any $x,y$ satisfying \eqref{eq:xyrange},
\begin{equation*}\label{eq:Henriot-cor2-euler}
\begin{split}
 &\sum_{\substack{x<n\leq x+y\\Q(n)\neq0}}
 F(|Q_1(n)|,\ldots,|Q_k(n)|)
 \prod_{j=1}^k|\Kl(a;|Q_j(n)|)|^{\nu_j}\\
 &\quad\ll 
 \Delta_S(a,\bnu,F)\,y
 \prod_{g<p\leq x}\left(1-\frac{\rho(p)}p\right)
 \prod_{\substack{p\le x\\p\nmid S}}
 \left(
 1+
 \sum_{j=1}^k
 G_F(\mathbf p_j^1)
 \frac{\rho_{Q_j}(p)
       \lambda_{a,\nu_j}(p)}{p}
 \right),
\end{split}
\end{equation*}
where \(S:= \prod_{p\in \mathcal{S}} p\).  Moreover, we have
\begin{equation}\label{eq:vartheta-bad}
1\le  \Delta_S(a,\bnu,F)\le \left(\frac{S}{\varphi(S)}\right)^C.
\end{equation}
The implied $\ll$-constant has the same dependencies as in Theorem~\ref{thm:main} and the constant \(C\) is the same as the one in Corollary~\ref{cor:discriminant}. 
\end{corollary}

Under a slightly stronger condition on \(\varepsilon\), the same
Euler product controls the corresponding sum over primes. We refer to \cite[Theorem~3]{NT-1998} of Nair--Tenenbaum, and \cite[Theorem~7]{Henriot-2012} of Henriot for the technical argument that allows us to prove the following.

\begin{theorem}\label{thm:prime-variable}
Under the same assumptions as in Theorem~\ref{thm:main}, and suppose that
\(Q(0)\ne0\), and 
\begin{equation}\label{eq:prime-epsilon-range}
 0<\varepsilon
 \leq\frac{\alpha}
 {100(g+1)(g+1+1/\delta)},
\end{equation}
we have, uniformly for \(x,y\) in the same range as \eqref{eq:xyrange}, 
\begin{equation}\label{eq:prime-variable}
\begin{split}
 &\sum_{\substack{x<p\leq x+y\\Q(p)\ne0}}
 F(|Q_1(p)|,\ldots,|Q_k(p)|)
 \prod_{j=1}^k|\Kl(a;|Q_j(p)|)|^{\nu_j}\\
 &\quad\ll
 \frac{|Q(0)|}{\varphi(|Q(0)|)}
 \frac{y}{\log x}\,
 \Delta_S(a,\bnu,F)
 \prod_{g<p\leq x}\left(1-\frac{\rho(p)}p\right)
 \prod_{\substack{p\leq x\\p\nmid S}}
 \left(
 1+\sum_{j=1}^k
 G_F(\mathbf p_j^1)
 \frac{\rho_{Q_j}(p)\lambda_{a,\nu_j}(p)}p
 \right).
\end{split}
\end{equation}
The constant \(c_0\) and the implied constant have the same
dependencies as in Theorem~\ref{thm:main}.
\end{theorem}

For \(\nu>-1\), define
\begin{equation}\label{eq:sato-moment}
 \sn(\nu)
:=\frac2\pi\int_0^\pi|2\cos\theta|^\nu\sin^2\theta\,d\theta
 =\frac{2^\nu\Gamma((\nu+1)/2)}
 {\sqrt\pi\,\Gamma((\nu+4)/2)}.
\end{equation}
In particular,
\[
 \sn(0)=1,\qquad \sn(1)=\frac8{3\pi},\qquad
 \sn(2)=1,\qquad \sn(4)=2.
\]

\begin{corollary}\label{cor:normalized}
Let $0<\alpha, \delta<1$, \(k\in \N\) and 
\(z_j>0\) be any fixed number for \(1\leq j\leq k\). Under the assumption of Theorem \ref{thm:main} for $Q_j$ \((1\le j\le k)\),  
\begin{equation*}\label{eq:normalized-moment}
 \sum_{\substack{x<n\leq x+y\\Q(n)\neq0}}
 \prod_{j=1}^k
 z_j^{\omega(|Q_j(n)|)}
 |\Kl(a;|Q_j(n)|)|^{\nu_j}
 \ll
 y(\log x)^{\sum_{j=1}^k(z_j\sn(\nu_j)-1)}
\end{equation*}
holds uniformly in
\(a\in\Z\setminus\{0\}\), for any \(x\geq c_0\|Q\|^\delta\) and \(x^\alpha<y\leq x\), where the constant $c_0$  and the implied constant  depend on $\bnu,\mathbf z,Q_1,\ldots,Q_k,\alpha,\delta$ only.
\end{corollary}

\begin{remark}
Theorem \ref{thm:main} and Corollaries \ref{cor:discriminant} and
\ref{cor:Henriot-cor2} are uniform in \(D\) in the precise sense stated there.
Corollary \ref{cor:normalized} is uniform in \(a\), but its proof uses
Erd\H{o}s's estimate \cite[Lemma~7]{Erdos-1952},
\[
 \sum_{p\leq x}\frac{\rho_{Q_j}(p)}p
 =\log\log x+O_{Q_j}(1),
\]
so the passage from the explicit Euler product to a pure power of
\(\log x\) introduces a further \(Q_j\)-dependent Mertens constant.  No
discriminant-uniform claim is made for that last simplification without an
additional effective prime-ideal theorem.
\end{remark}

We now give a lower bound for the weighted moments over \((x/2,x]\).
The restriction to squarefree integers coprime to \(2a\) will also be
useful in the applications.

\begin{theorem}\label{thm:lower}
Let \(z,\nu,A>0\) be fixed. For all sufficiently large \(x\) in terms of
\(A,z,\nu\), uniformly for nonzero integers \(a\) with \(|a|\leq x^A\),
\begin{equation*}\label{eq:lower-bound}
 \sum_{\substack{x/2<n\leq x\\n\ \mathrm{squarefree}\\(n,2a)=1}}
 z^{\omega(n)}|\Kl(a;n)|^\nu
 \gg_{A,z,\nu}
 \prod_{p\mid a}\left(1+\frac{z\sn(\nu)}p\right)^{-1}
 x(\log x)^{z\sn(\nu)-1}.
\end{equation*}
\end{theorem}

\section{Applications}\label{sec:applications}

\subsection{Moments of Kloosterman sums}

We first record a consequence of Corollary~\ref{cor:normalized}
and Theorem~\ref{thm:lower} in which the dependence on \(a\)
is explicit.

\begin{corollary}\label{cor:moment-order}
Let \(z>0\), \(\nu\geq1\), and \(A>0\) be fixed. Uniformly for
\(x\geq3\) and nonzero integers \(a\) with \(|a|\leq x^A\),
\[
 \prod_{p\mid a}\left(1+\frac{z\sn(\nu)}p\right)^{-1}
 \ll \frac{1}{x(\log x)^{z\sn(\nu)-1}}
 \sum_{n\leq x}z^{\omega(n)}|\Kl(a;n)|^\nu
 \ll 1,
\]
where the first implied constant depends only on \(A,z,\nu\), and
the second only on \(z,\nu\). The upper bound holds for every
nonzero integer \(a\), without the restriction \(|a|\leq x^A\).
\end{corollary}

For fixed \(a\ne0\), \(z>0\), and \(\nu\geq1\), the product
in Corollary~\ref{cor:moment-order} is a positive constant, so
\begin{equation}\label{eq:cor41-fixeda}
	\sum_{n\leq x}z^{\omega(n)}|\Kl(a;n)|^\nu
	\asymp x(\log x)^{z\sn(\nu)-1}.
\end{equation}
For \(\nu=1\), the choices \(z=1/2\) and \(z=1\) correspond
to the two sums studied by Fouvry and Michel. Since
\(\sn(1)=8/(3\pi)\), \eqref{eq:cor41-fixeda} yields
\begin{equation}\label{eq:FM-moments}
 \sum_{n\leq x}\frac{|\Kl(a;n)|}{2^{\omega(n)}}
 \asymp \frac{x}{(\log x)^{1-\frac{4}{3\pi}}},
 \qquad
 \sum_{n\leq x}|\Kl(a;n)|
 \asymp \frac{x}{(\log x)^{1-\frac{8}{3\pi}}}.
\end{equation}

Let us review the previous results for \(a=1\). Fouvry and Michel
\cite[Theorems~1.1 and~1.2]{FM-2003}, as corrected in
\cite{FM-2006-corrigendum}, proved that
\[
 \sum_{n\leq x}\frac{|\Kl(1;n)|}{2^{\omega(n)}}
 \ll x\bigg(\frac{\log\log x}{\log x}\bigg)^{1-\frac{4}{3\pi}},\quad
 \sum_{n\leq x}|\Kl(1;n)|
 \ll\frac{x(\log\log x)^{2-\frac{8}{3\pi}}}
              {(\log x)^{1-\frac{8}{3\pi}}}.
\]
Their Theorem~1.3 shows that there is an absolute constant
\(\delta>5/12\) such that 
\[
 \frac{x}{\log x}\exp\bigl((\log\log x)^\delta\bigr)
 \ll\sum_{n\leq x}\frac{|\Kl(1;n)|}{2^{\omega(n)}}
 \leq\sum_{n\leq x}|\Kl(1;n)|.
\]
The result in  \eqref{eq:FM-moments} supersedes  Fouvry and Michel's bounds with a precise order of magnitude.

Another interesting consequence of \eqref{eq:cor41-fixeda} is its spectral application  in Section~\ref{subsec:spectral} below. For this purpose, we use the second moment. Since \(\sn(2)=1\), the case \(z=1\) of
\eqref{eq:cor41-fixeda} reduces to
\begin{equation}\label{eq:cor41-fixeda-2}
	\sum_{n\leq x}|\Kl(a;n)|^2\asymp x,
\end{equation}
for each fixed \(a\ne0\).
For \(a=1\), the earlier estimates of Fouvry and Michel
\cite{FM-2003}, recorded in \cite[Theorem~6.3]{Michel-2006}, are
\[
 \frac{x}{\log x}\exp\!\left((\log\log x)^{\frac{5}{17}}\right)
 \ll \sum_{n\leq x}|\Kl(1;n)|^2
 \ll x(\log\log x)^3.
\]
The order \(x\) determined in \eqref{eq:cor41-fixeda-2} matches the  prediction by the probabilistic model in \cite[p.~234]{Michel-2006}.

\subsection{Large values of Kloosterman sums}

The matching upper and lower bounds on the moments determine the logarithmic frequency of
large Kloosterman sums. For \(B\geq0\), define
\begin{equation*}\label{eq:Kl-rate-function}
 I(B):=\sup_{\nu>0}\{\nu B-\sn(\nu)+1\}.
\end{equation*}
This is the Legendre--Fenchel formula for the upper tail associated
with the logarithmic moment function \(\sn(\nu)-1\); see
\cite[Section~2.3]{DemboZeitouni-2010}. By
\eqref{eq:sato-moment}, \(\sn\) is strictly convex, with
\(\sn'(0)=-1/2\) and \(\sn'(\nu)\to\infty\) as \(\nu\to\infty\).
The supremum is therefore attained at the unique \(\nu_B>0\)
satisfying \(\sn'(\nu_B)=B\), and
\[
 I(B)=1+B\nu_B-\sn(\nu_B),\qquad I'(B)=\nu_B>0.
\]

For \(x\geq3\), let \(\mathbb P_x\) be the uniform probability
measure on the integers \(n\in(x/2,x]\).
For \(a\ne0\) and \(B\geq0\), set
\[
 \mathcal L_{a,B}(x)
 :=\{n\in(x/2,x]:|\Kl(a;n)|\geq(\log x)^B\}.
\]

\begin{corollary}\label{cor:Kl-large-values}
Let \(a\ne0\) and \(B\geq0\) be fixed. Then
\[
 \lim_{x\to\infty}
 \frac{\log\mathbb P_x\!\left(\mathcal L_{a,B}(x)\right)}{\log\log x}
 =-I(B).
\]
\end{corollary}

Varying \(z\) in the weighted moments
\eqref{eq:cor41-fixeda} determines the typical number of prime factors under the same large-value condition.

\begin{corollary}\label{cor:Kl-large-value-concentration}
Let \(a\ne0\) and \(B\geq0\) be fixed. For every fixed
\(0<\delta<\sn(\nu_B)\), as \(x\to\infty\),
\[
 \mathbb P_x\!\left(
 \left.
 \left|\frac{\omega(n)}{\log\log x}-\sn(\nu_B)\right|
 \geq\delta\ \right|\ \mathcal L_{a,B}(x)\right)
 \ll_{a,B,\delta} (\log x)^{-\frac{\delta^2}{8\sn(\nu_B)}}.
\]
Thus \(\omega(n)/\log\log x\) converges in probability to
\(\sn(\nu_B)\) under the conditional law
\(\mathbb P_x(\,\cdot\mid\mathcal L_{a,B}(x))\).
\end{corollary}

For example, since
\[
 \sn(2)=1,\qquad \sn'(2)=\frac{1}{4},\qquad
 I\left(\frac{1}{4}\right)=\frac{1}{2},
\]
we have \(\nu_{1/4}=2\) and
\(\mathbb P_x(\mathcal L_{a,1/4}(x))=(\log x)^{-1/2+o(1)}\).
Corollary~\ref{cor:Kl-large-value-concentration} gives
\[
 \frac{\omega(n)}{\log\log x}\longrightarrow 1
\]
in probability under the conditional law
\(\mathbb P_x(\,\cdot\mid\mathcal L_{a,1/4}(x))\). Thus the typical number of prime factors remains \(\log\log x\), even under this large-value condition.

The upper-tail estimate has the following consequence for classical
Kloosterman sums. Sarnak \cite[p.~33]{Sarnak-1990} noted that, as pointed out by Serre, it was not known whether
\(S(a,b;n)=O(n^{1/2-\varepsilon})\) for some
\(\varepsilon>0\) could hold for fixed nonzero integers \(a,b\) as
\(n\to\infty\).
Fouvry and Michel \cite[Corollary~1.4]{FM-2003} excluded such a bound
for \(a=b=1\) with the estimate
\[
 S(1,1;n)=\Omega\!\left(
 \sqrt n\,\frac{\exp\bigl((\log\log n)^{\frac{5}{12}}\bigr)}
 {\log n}\right).
\]
We have the following stronger consequence.

\begin{corollary}\label{cor:Kl-omega}
Let \(a,b\) be fixed nonzero integers. For every fixed \(B>0\),
\[
 S(a,b;n)=\Omega\bigl(\sqrt n\,(\log n)^B\bigr).
\]
\end{corollary}

\subsection{Hecke eigenvalues and Kloosterman sums}

Let \(f\) be a holomorphic primitive form\footnote{Here, a primitive form is a  normalized cuspidal Hecke eigenform.} of weight
\(\kappa\) for \(\mathrm{SL}_2(\mathbb Z)\), with Fourier expansion
\[
 f(\tau)=\sum_{n\geq1}\lambda_f(n)n^{\frac{\kappa-1}{2}}\e(n\tau),
 \qquad \lambda_f(1)=1.
\]
The coefficients \(\lambda_f(n)\) are real. By the Sato--Tate theorem
\cite{BLGHT-2011} and partial summation (see also \cite[Lemma 5.2]{Mangerel-2023}), for any fixed \(\mu>0\),
\begin{equation}\label{eq:Hecke-ST-prime}
 \sum_{p\leq x}\frac{|\lambda_f(p)|^\mu}{p}
 =\bigl(\sn(\mu)+o(1)\bigr)\log\log x.
\end{equation}
Applying Corollary~\ref{cor:Henriot-cor2} with \(Q(X)=X\) and
\(F(n)=|\lambda_f(n)|^\mu\) gives the following estimate.

\begin{corollary}\label{cor:Hecke-Kl-mixed}
Let \(f\) be a holomorphic primitive form as above, and let \(\mu,\nu>0\), \(0<\alpha<1\), and
\(\varepsilon>0\) be fixed. Uniformly for \(a\in\Z\setminus\{0\}\) and
\(x^\alpha<y\leq x\),
\[
 \sum_{x<n\leq x+y}
 |\lambda_f(n)|^\mu|\Kl(a;n)|^\nu
 \ll
 y(\log x)^{\sn(\mu)\sn(\nu)-1+\varepsilon},
\]
where the implied constant depends only on
\(f,\mu,\nu,\alpha,\varepsilon\).
\end{corollary}

Xi \cite[Conjecture~2.2]{Xi-2020} conjectured that, for every fixed
primitive cusp form \(f\), holomorphic or Maass,
\begin{equation*}\label{eq:Xi-integer-correlation}
 \sum_{n\leq x}\lambda_f(n)\Kl(1;n)
 \ll_f\frac{x}{(\log x)^{2018}}.
\end{equation*}
Taking \(\mu=\nu=1\) in Corollary~\ref{cor:Hecke-Kl-mixed}, the dyadic argument yields the following bound.

\begin{corollary}
Let \(f\) be a holomorphic primitive form as above, and let \(\varepsilon>0\) be fixed. Uniformly for \(a\in\Z\setminus\{0\}\),
\begin{equation*}\label{eq:Hecke-Kl-correlation}
 \sum_{n\leq x}|\lambda_f(n)\Kl(a;n)|
 \ll
 \frac{x}{(\log x)^{1-\frac{64}{9\pi^2}-\varepsilon}},
\end{equation*}
where the implied constant depends only on \(f\) and \(\varepsilon\).
\end{corollary}

Applying Theorem~\ref{thm:prime-variable} with
\(Q(X)=X+h\), \(F(n)=|\lambda_f(n)|\), and \(\nu=1\), and then using the Euler-product calculation in the proof of Corollary~\ref{cor:Hecke-Kl-mixed}, we get the following estimate.

\begin{corollary}\label{cor:Hecke-Kl-prime-shift}
Let \(f\) be a holomorphic primitive form as above, let \(h\ne0\) be a fixed integer, and let
\(\varepsilon>0\) be fixed. Uniformly for
\(a\in\Z\setminus\{0\}\),
\[
 \sum_{\substack{p\leq x\\p+h>0}}
 |\lambda_f(p+h)\Kl(a;p+h)|
 \ll
 \frac{x}{(\log x)^{2-\frac{64}{9\pi^2}-\varepsilon}},
\]
where the implied constant depends only on \(f,h,\varepsilon\).
\end{corollary}

Xi also conjectured \cite[Conjecture~2.1]{Xi-2020} that
\begin{equation}\label{eq:Xi-prime-correlation}
 \sum_{p\leq x}\lambda_f(p)\Kl(1;p)
 =o\left(\frac{x}{\log x}\right).
\end{equation}
It must be noted that Corollary~\ref{cor:Hecke-Kl-prime-shift} does {\it not} cover the conjecture \eqref{eq:Xi-prime-correlation} due to the required hypothesis \(Q(0)\ne0\) in Theorem~\ref{thm:prime-variable}, which  excludes \(h=0\). Neither Xi's conjecture
nor Katz's Problem~III \cite[p.~62]{Xi-2020} is  addressed here.

Combining the mixed-moment bound with Theorem~\ref{thm:lower}, we also get the following result.

\begin{corollary}\label{cor:Hecke-Kl-noncoincidence}
Let \(f\) be as above, and let \(a\ne0\) be fixed. For any $\varepsilon>0$, we have
\[
 \#\left\{n\leq x: \mu^2(n)=1,\, (n,2a)=1,\,
 \lambda_f(n)\ne\pm\Kl(a;n)\right\}
 \gg \frac{x}{(\log x)^\varepsilon},
\]
where the implied constant depends only on \(f,a,\varepsilon\).
\end{corollary}

For a primitive Hecke--Maass cusp form \(g\) with trivial nebentypus, Xi \cite[Theorem~1.1]{Xi-2020} proved that
\[
\#\left\{n\leq x: \mu^2(n)=1,\,\omega(n)\leq 100,\,
\lambda_f(n)\ne\pm\Kl(1;n)\right\}
\gg \frac{x}{\log x}.
\]
Corollary~\ref{cor:Hecke-Kl-noncoincidence} concerns the holomorphic case (but not the maass case in view of the unsettled generalized Ramanujan conjecture), and gives a quantitative result with no restriction on \(\omega(n)\).

\subsection{Number variance on the modular surface}\label{subsec:spectral}

Li and Sarnak's formula for the number variance on the modular
surface involves the Kloosterman sums \(S(u,u;n)\).
When \((u,n)=1\), the identity
\(S(u,u;n)=\sqrt n\,\Kl(u^2;n)\) relates their second moments
to those considered above. Li and Sarnak
\cite[(1.13)]{LiSarnak-2004}
conjectured that, for each fixed nonzero integer \(u\),
\begin{equation}\label{eq:LS-conjecture}
 \sum_{n\leq x}\frac{|S(u,u;n)|^2}{n^2}
 \sim c_{u}\log x,
\end{equation}
with a positive constant \(c_{u}\). We confirm the order of magnitude -- a weaker form of their conjecture, below.   

\begin{theorem}\label{thm:LS}
Let \(A>0\) be fixed. There exists a constant $x_0(A)$ depending on $A$ such that uniformly for all $x\ge x_0(A)$ and  nonzero integers \(u\) with \(|u|\leq x^A\),
\[
 \prod_{p\mid u}\left(1+\frac{1}{p}\right)^{-1}
 \ll \frac{1}{\log x}\sum_{n\leq x}\frac{|S(u,u;n)|^2}{n^2}
 \ll \tau(|u|)^2,
\]
where \(\tau\) is the divisor function, the first implied constant
depends only on \(A\), and the second is absolute. The upper bound
holds for every nonzero integer \(u\).
\end{theorem}
\begin{remark*}
    The case $u=0$ cannot be included as \(S(0,0,n)= \varphi(n)\). 
\end{remark*}

We now use our moment
bounds and the variance formula of Li and Sarnak
\cite[Theorem~1.4]{LiSarnak-2004} to estimate the smoothed
number variance and the remainder in Weyl's law. See also \cite[\S11.1]{Michel-2006}.

Let \(1/4+t_j^2\;(t_j\geq0)\) denote the cuspidal Laplace
eigenvalues on \(\mathrm{SL}(2,\Z)\backslash\mathbb H\). We put \(N(t)=\#\{j:0\leq t_j\leq t\}\), counted
with multiplicity. As in \cite[(1.2)]{LiSarnak-2004}, we define the remainder \(S(t)\) by
\begin{equation}\label{eq:spectral-Weyl}
 N(t)=\frac{t^2}{12}-\frac{2t\log t}{\pi}
 +\frac{2+\log\pi-\log2}{\pi}t+\frac{13}{144}+S(t).
\end{equation}

Let \(h\) be a fixed real even Schwartz function with
\(\widehat h(0)=1\) and \(\operatorname{supp}\widehat h\subset(-1,1)\), where
\(\widehat h(\xi)=\int_{\R}h(x)\e(-x\xi)\,dx\) is its Fourier transform.
Let \(\psi\in C_c^\infty((1,2))\) be nonnegative, with
\(\int_0^\infty\psi(x)\,dx=1\). Define the smoothed counting
function and its variance by
\[
 N_h(t,L)=\sum_j h(L(t-t_j)),
 \qquad
 \Sigma_h(T,L)=\frac{1}{T}\int_0^\infty\psi(t/T)
 \left(N_h(t,L)-\frac{t}{6L}\right)^2dt.
\]
For fixed \(0<\delta<1/121\), consider the range
\begin{equation}\label{eq:spectral-window}
 \frac{1+\delta}{\pi}\log T\leq L
 \leq\frac{122}{121\pi}\log T.
\end{equation}

\begin{corollary}\label{cor:spectral-variance}
Let \(h,\psi,\delta\) be as above, with \(\widehat h\not\equiv 0\) in  \(((1+\delta)^{-1},1)\). Then uniformly for \(L\) in the range
\eqref{eq:spectral-window}, we have
\begin{equation}\label{eq:spectral-variance-order}
 \Sigma_h(T,L)\asymp\frac{T}{L}.
\end{equation}
Moreover, 
\begin{equation}\label{eq:spectral-mean-square}
 \int_T^{2T}S(t)^2\,dt\gg\frac{T^2}{\log T}.
\end{equation}
In particular,
\[
 S(T)=\Omega \bigg(\sqrt{\frac{T}{\log T}}\bigg),
\]
where the implied constants in these estimates depend only on \(h,\psi,\delta\).
\end{corollary}

Note that the nonvanishing condition holds, for example, if \(\widehat h=1\)
on \([-b,b]\) for some \((1+\delta)^{-1}<b<1\).

Li and Sarnak \cite[p. 5]{LiSarnak-2004}
observed that, under \eqref{eq:LS-conjecture}, the number variance
would be asymptotic to a positive multiple of \(T/L\), as
predicted by a Poisson model for the spectrum. Our estimate \eqref{eq:spectral-variance-order} establishes the order
of magnitude \(T/L\) throughout the range
\eqref{eq:spectral-window}.

Previously, Selberg's estimates \cite[(1.5)--(1.6)]{LiSarnak-2004} are
\[
 \int_T^{2T}S(t)^2\,dt
 \gg\frac{T^2}{(\log T)^2},
 \qquad
 S(T)=\Omega \bigg(\frac{\sqrt T}{\log T}\bigg).
\]
Improving his lower bound remains a challenging problem, with only modest progress to date.
Li and Sarnak \cite[Theorem~1.1]{LiSarnak-2004} obtained the first
improvement by a factor $\exp((\log \log T)^{5/17})$,  which is still $o(\log n)$,  
\[
\begin{aligned}
 \int_T^{2T}S(t)^2\,dt
 &\gg\frac{T^2}{(\log T)^2}
 \exp\!\Big((\log\log T)^{\frac{5}{17}}\Big),\\
 S(T)&=\Omega\bigg(\frac{\sqrt T}{\log T}
 \exp \Big(\frac{1}{2}(\log\log T)^{\frac{5}{17}}\Big)\bigg).
\end{aligned}
\]
The lower bound \eqref{eq:spectral-mean-square}  represents further progress toward Selberg's  mean-square estimate, giving a bound of order $T^2/\log T$. While this improvement may appear modest,  Li and Sarnak  \cite[p. 5]{LiSarnak-2004} suggested that it may be the true order.

\subsection{Shifted moments of Kloosterman sums}

El Abdalaoui, Shparlinski, and Steiner
\cite[Conjecture~2.2(2)]{ASS-2024} conjectured that, for
each fixed integer \(a\ne0\), fixed integers
\(0\leq h_1<\cdots<h_s\), and fixed positive even integers
\(\nu_1,\ldots,\nu_s\), one has
\begin{equation}\label{eq:ASS-conj}
\frac{1}{x}\sum_{n\leq x}
\prod_{j=1}^s
\left(
\frac{|\Kl(a;n+h_j)|}{2^{\omega(n+h_j)}}
\right)^{\nu_j}
\ll
\bigg(
\frac{1}{x}\sum_{n\leq x}
\frac{|\Kl(a;n)|}{2^{\omega(n)}}
\bigg)^{\nu_1+\cdots+\nu_s}.
\end{equation}

This conjecture leads us to consider these shifted moments
for arbitrary positive real exponents. The following bound
is a consequence of
Corollary~\ref{cor:normalized} with \(Q_j(X)=X+h_j\).

\begin{corollary}\label{cor:ASS}
Let \(0\leq h_1<\cdots<h_s\) be fixed integers, and let
\(\nu_1,\ldots,\nu_s>0\) be fixed. Uniformly for
\(a\in\Z\setminus\{0\}\), we have
\[
 \sum_{n\leq x}
 \prod_{j=1}^s
 \left(\frac{|\Kl(a;n+h_j)|}{2^{\omega(n+h_j)}}\right)^{\nu_j}
 \ll x(\log x)^{\sum_{j=1}^s2^{-\nu_j}\sn(\nu_j)-s},
\]
where the implied constant depends only on \(\boldsymbol\nu,\mathbf h\).
\end{corollary}

The conjectured inequality \eqref{eq:ASS-conj} has already failed for \(s=1\),
\(h_1=0\), and \(\nu_1=2\). Indeed, \eqref{eq:cor41-fixeda} implies, for fixed $a\neq 0$,
\begin{equation*}
 \frac{\displaystyle\frac{1}{x}\sum_{n\leq x}
       \bigg(\frac{|\Kl(a;n)|}{2^{\omega(n)}}\bigg)^2}
      {\displaystyle\bigg(\frac{1}{x}\sum_{n\leq x}
       \frac{|\Kl(a;n)|}{2^{\omega(n)}}\bigg)^2}
 \asymp(\log x)^{\frac{5}{4}-\frac{8}{3\pi}}
 \longrightarrow\infty
 \qquad(x\to\infty), 
\end{equation*}
which would have been $O(1)$ if \eqref{eq:ASS-conj} was true.

For independent nonnegative random variables $X_1,\ldots, X_s$, mixed moments
factor into products of the individual moments of the corresponding
orders; see \cite[Theorem~2.1.13]{Durrett-2019}. One still has $\mathbb{E}( \prod_{i=1}^s X_i^{\nu_i})\le \prod_{i=1}^s \mathbb{E}(X_i^{\nu_i})$ if the random variables are negatively associated\footnote{The intuition behind negatively associated nonnegative random variables is that these random variables are not simultaneously large.}, rather than independent.
To interpret probabilistically the normalized Kloosterman values at
distinct shifts, we are therefore led to the following inequality instead of
\eqref{eq:ASS-conj}: for fixed \(a\ne0\), fixed integers
\(0\leq h_1<\cdots<h_s\), and fixed real
\(\nu_1,\ldots,\nu_s\geq1\), one would expect
\begin{equation}\label{eq:ASS-moment-product}
 \frac{1}{x}\sum_{n\leq x}
 \prod_{j=1}^s
 \left(\frac{|\Kl(a;n+h_j)|}{2^{\omega(n+h_j)}}\right)^{\nu_j}
 \ll
 \prod_{j=1}^s
 \bigg(
 \frac{1}{x}\sum_{n\leq x}
 \left(\frac{|\Kl(a;n)|}{2^{\omega(n)}}\right)^{\nu_j}
 \bigg).
\end{equation}
Now, Corollary~\ref{cor:ASS} gives the left-hand side of  \eqref{eq:ASS-moment-product} the upper bound \((\log x)^{\sum_{j=1}^s(2^{-\nu_j}\sn(\nu_j)-1)}\) that is dominated by the right-hand side by  \eqref{eq:cor41-fixeda}, verifying \eqref{eq:ASS-moment-product}.

\section{Local averages and weighted congruence densities}

We first estimate the local Kloosterman moments, whose prime averages
determine the logarithmic exponents in the corollaries. We then prove the
multiplicativity of the weighted density, compute it at primes outside
\(\mathcal S\), and bound it at the remaining primes not dividing \(D\).

By \eqref{eq:Kl-Weil}, which follows from the Weil--Estermann estimates \cite[(1.1)]{ASS-2024}, we have
\begin{equation}\label{eq:prime-power-Kl-bound}
 |\Kl(b;p^\ell)|\leq
 \begin{cases}
  2,&p>2,\\
  2\kappa(\ell)\le 2^{5/2},&p=2,
 \end{cases}
 \qquad(b\in\Z,\ \ell\geq1).
\end{equation}
For fixed \(\nu\geq0\), define the local averages
\begin{equation}\label{eq:lambda}
 \lambda_{a,\nu}(p^\ell)
 :=
 \frac1{\varphi(p^\ell)}
 \sideset{}{^*}\sum_{h\bmod p^\ell}
 |\Kl(a\bar h^{2};p^\ell)|^\nu
 \qquad(\ell\geq1).
\end{equation}
We also set \(\lambda_{a,\nu}(1)=1\) by convention. 
\begin{lemma}\label{lem:lambda}
Let \(\nu\geq0\) be fixed. For \(\ell \geq 1\), \(\lambda_{a,\nu}(2^\ell)\leq 2^{\frac52\nu}\) and  
\begin{equation}\label{eq:lambda-power}
 \lambda_{a,\nu}(p^\ell)\leq  2^\nu
 \qquad(p \ \textrm{odd}).
\end{equation}
If \(p\nmid2a\), then
\begin{equation}\label{eq:lambda-prime}
 \lambda_{a,\nu}(p)=\sn(\nu)+O_\nu(p^{-1/8}),
\end{equation}
uniformly in \(a\bmod p\). For a prime \(p\) and an integer \(a\), we also have
\begin{equation}\label{eq:lambda-prime-upper}
 \lambda_{a,\nu}(p)\leq\sn(\nu)+O_\nu(p^{-1/8}).
\end{equation}
\end{lemma}

\begin{proof}
  Averaging \eqref{eq:prime-power-Kl-bound} over the reduced residue
  classes proves \eqref{eq:lambda-power}.

  The upper bound \eqref{eq:lambda-prime-upper} follows from
  \eqref{eq:lambda-prime} when \(p\nmid2a\). If \(p\mid a\), then
  \(\Kl(a\bar h^2;p)=-p^{-1/2}\) for \(h\in\F_p^\times\), so
  \(\lambda_{a,\nu}(p)=p^{-\nu/2}\). For \(\nu>0\), this is at most
  \(\sn(\nu)\) for all sufficiently large \(p\); the finitely many
  remaining primes are absorbed in the error constant. The case
  \(\nu=0\) is exact, and \(p=2\) is covered by
  \eqref{eq:lambda-power}.

  It remains to prove \eqref{eq:lambda-prime}. Suppose that
  \(p\nmid2a\). By \eqref{eq:prime-power-Kl-bound}, we may write
  \[
  \Kl(t;p)=2\cos\theta_{p,t},
  \qquad
  \theta_{p,t}\in[0,\pi]
  \qquad(t\in\F_p^\times).
  \]
  For \(r\geq0\), put
  \[
  \operatorname{sym}_r(\theta)=\frac{\sin((r+1)\theta)}{\sin\theta}.
  \]

  For \(r\geq1\), Katz \cite[Example~13.6]{Katz-1988} proved
  \begin{equation}\label{eq:full-symmetric-power}
   \bigg|\sum_{t\in\F_p^\times}
   \operatorname{sym}_r(\theta_{p,t})\bigg|
   \leq \frac12(r+1)p^{1/2}.
  \end{equation}
  Fouvry and Michel \cite[Lemma~2.1]{FM-2003} established the corresponding
  estimate for square parameters
  \begin{equation}\label{eq:square-symmetric-power}
   \bigg|\sum_{h\in\F_p^\times}
   \operatorname{sym}_r(\theta_{p,\bar h^{2}})\bigg|
   \leq 3(r+1)p^{1/2}.
  \end{equation}

  If \(a\) is a quadratic residue modulo \(p\), multiplication by
  \(a\) permutes the nonzero quadratic residues. Thus
  \[
  \sum_{h\in\F_p^\times}
  \operatorname{sym}_r(\theta_{p,a\bar h^{2}})
  =\sum_{h\in\F_p^\times}
  \operatorname{sym}_r(\theta_{p,\bar h^{2}}).
  \]
  Otherwise, \(h\mapsto\bar h^2\) and \(h\mapsto a\bar h^2\) are
  two-to-one maps onto the nonzero quadratic residues and nonresidues,
  respectively.
  Hence
  \[
  \sum_{h\in\F_p^\times}
  \operatorname{sym}_r(\theta_{p,a\bar h^{2}})
  +\sum_{h\in\F_p^\times}
  \operatorname{sym}_r(\theta_{p,\bar h^{2}})
  =
  2 \sum_{t\in\F_p^\times}
  \operatorname{sym}_r(\theta_{p,t}).
  \]
  It follows from \eqref{eq:full-symmetric-power} and
  \eqref{eq:square-symmetric-power} that, in either case,
  \begin{equation}\label{eq:a-symmetric-power}
  \begin{aligned}
   \bigg|\sum_{h\in\F_p^\times}
   \operatorname{sym}_r(\theta_{p,a\bar h^{2}})\bigg|
   &\leq \bigg|\sum_{h\in\F_p^\times}
   \operatorname{sym}_r(\theta_{p,\bar h^{2}})\bigg|
   +2\bigg|\sum_{t\in\F_p^\times}
   \operatorname{sym}_r(\theta_{p,t})\bigg|\\
   &\leq 4(r+1)p^{1/2},
  \end{aligned}
  \end{equation}
  uniformly in \(a\in\F_p^\times\). The argument of
  \cite[Lemmas~2.2 and 2.3]{FM-2003}, together with
  \eqref{eq:a-symmetric-power}, implies
  \begin{equation}\label{eq:a-ST-discrepancy}
    \sup_{0\leq u<v\leq\pi}
    \left|
    \frac1{p-1}
    \#\left\{
    h\in\F_p^\times:
    u\leq\theta_{p,a\bar h^{2}}\leq v
    \right\}
    -
    \frac2\pi\int_u^v\sin^2t\,dt
    \right|
    \ll p^{-1/8},
  \end{equation}
  uniformly in \(a\in\F_p^\times\), with an absolute implied constant.

  For \(0\leq t\leq\pi\), put
  \[
  F_{p,a}(t)=\frac1{p-1}
  \#\{h\in\F_p^\times:\theta_{p,a\bar h^2}\leq t\}.
  \]
  For fixed \(\nu>0\), integration by parts and
  \eqref{eq:a-ST-discrepancy} imply
  \[
  \begin{aligned}
    \frac1{p-1}\sum_{h\in\F_p^\times}
    |2\cos\theta_{p,a\bar h^2}|^\nu
    &=\int_0^\pi |2\cos t|^\nu\,dF_{p,a}(t)\\
    &=\frac2\pi\int_0^\pi |2\cos t|^\nu\sin^2t\,dt
    +O_\nu(p^{-1/8})\\
    &=\sn(\nu)+O_\nu(p^{-1/8}).
  \end{aligned}
  \]
  This proves \eqref{eq:lambda-prime} and completes the proof.
\end{proof}

\begin{lemma}\label{lem:lifting}
Let \(p\) be a prime, and let \(R\in\Z[X]\) be nonconstant. Suppose
\(p\nmid\lc(R)\disc(R)\). For any integer \(\ell\geq1\) and real
\(\nu\geq0\),
\begin{equation*}
 \sum_{\substack{r\bmod p^{2\ell}\\p^\ell\parallel R(r)}}
 \left|
 \Kl \left(
 a\,\overline{R(r)/p^\ell}^{2};p^\ell
 \right)
 \right|^\nu
 =\rho_R(p)\varphi(p^\ell)\lambda_{a,\nu}(p^\ell).
\end{equation*}
\end{lemma}

\begin{proof}
Fix an integer \(u\) with \(p\nmid u\). The polynomials
\(R(X)-p^\ell u\) and \(R(X)\) have the same reduction modulo \(p\), and hence share the same roots mod $p$ 
which are all simple by the hypothesis
\(p\nmid\lc(R)\disc(R)\). By Hensel's lemma, each of these
\(\rho_R(p)\) roots, say $r$, lifts uniquely to a solution of
\(R(r)\equiv p^\ell u\,(\text{mod}\,p^{2\ell})\). Thus
\[
 \#\{r\bmod p^{2\ell}:R(r)\equiv p^\ell u\,(\text{mod}\,p^{2\ell})\}
 =\rho_R(p).
\]
Since \(p\nmid u\), the congruence \(R(r)\equiv p^\ell u\,(\text{mod}\,p^{2\ell})\) is equivalent to
\(p^\ell\parallel R(r)\) together with
\(R(r)/p^\ell\equiv u\,(\text{mod}\,p^\ell)\). For the same reason as in \eqref{p-property}, we rewrite the summation over
\(u\bmod p^\ell\) with \(p\nmid u\) to get
\[
\begin{aligned}
 \sum_{\substack{r\bmod p^{2\ell}\\p^\ell\parallel R(r)}}
 \left|\Kl\left(a\,\overline{R(r)/p^\ell}^{\,2};p^\ell\right)\right|^\nu
 &=\rho_R(p)\sideset{}{^*}\sum_{u\bmod p^\ell}
   |\Kl(a\bar u^{\,2};p^\ell)|^\nu\\
 &=\rho_R(p)\varphi(p^\ell)\lambda_{a,\nu}(p^\ell),
\end{aligned}
\]
by \eqref{eq:lambda}.
\end{proof}

\begin{lemma}\label{lem:H-properties}
The function \(H_{a,\bnu}:\N^k\to\R_{\geq0}\) is multiplicative,
in the sense that
\[
 H_{a,\bnu}(\mathbf{de})
 =H_{a,\bnu}(\mathbf d)H_{a,\bnu}(\mathbf e)
\]
when \((d_1\cdots d_k,e_1\cdots e_k)=1\), where \(\mathbf{d}=(d_1,\ldots, d_k)\), \(\mathbf{e}=(e_1,\ldots, e_k)\) and \(\mathbf{de}:=(d_1e_1,\ldots, d_ke_k)\).  Moreover,
\begin{align}\label{eq:H-comparison}
 H_{a,\bnu}(\mathbf d)
 & \le
 \frac{\rho(d_1\cdots d_k)}{d_1\cdots d_k} \prod_{j=1}^k \big(2^{\omega(d_j)}\kappa(v_2(d_j))\big)^{\nu_j}. 
\end{align}
In particular, we have  for \(p^{\boldsymbol\ell}=(p^{\ell_1},\ldots,p^{\ell_k})\) where \({\boldsymbol\ell}=(\ell_1,\ldots,\ell_k)\),
\begin{equation}\label{eq:H-fair-2}
    H_{a,\bnu}(2^{\boldsymbol\ell})\le    2^{\frac52|{\boldsymbol\nu}|} 
    \frac{\rho_{Q}(2^{|{\boldsymbol\ell}|})}{2^{|{\boldsymbol\ell}|}}
    \qquad \mbox{ and }\qquad 
    H_{a,\bnu}(p^{\boldsymbol\ell})\le    2^{|{\boldsymbol\nu}|} 
    \frac{\rho_{Q}(p^{|{\boldsymbol\ell}|})}{p^{|{\boldsymbol\ell}|}}\quad (p \ \textrm{odd}),
  \end{equation}
and for \(\mathbf{d}=\mathbf p_j^\ell:=(1,\ldots,1,p^\ell,1,\ldots,1)\), 
\begin{equation}\label{eq:H-fair}
    H_{a,\bnu}(\boldsymbol{2}_j^\ell)\le    2^{\frac52\nu_j}
    \frac{\rho_{Q_j}(2^\ell)}{2^\ell}
    \qquad \mbox{ and }\qquad 
    H_{a,\bnu}(\mathbf p_j^\ell)
    \le  2^{\nu_j}
    \frac{\rho_{Q_j}(p^\ell)}{p^\ell} \quad (p \ \textrm{odd}).
  \end{equation}  
\end{lemma}

\begin{proof}
Suppose that \((d_1\cdots d_k,e_1\cdots e_k)=1\). Then
\[
L(\mathbf d\mathbf e)=L(\mathbf d)L(\mathbf e),
\qquad (L(\mathbf d),L(\mathbf e))=1.
\]
Furthermore,
\[
\mathcal P_{\mathbf Q}(\mathbf d\mathbf e,r)
\quad\Longleftrightarrow\quad
\mathcal P_{\mathbf Q}(\mathbf d,r)
\ \land\
\mathcal P_{\mathbf Q}(\mathbf e,r).
\]
Hence, by the Chinese remainder theorem, the admissible residue classes modulo
\(L(\mathbf d\mathbf e)\) are in bijection with pairs of admissible residue
classes modulo \(L(\mathbf d)\) and \(L(\mathbf e)\). We can thus split the summation in \eqref{def-H} for \(H_{a,\bnu}(\mathbf{de})\).

For \(p\mid d_1\cdots d_k\), we have
\(v_p(d_je_j)=v_p(d_j)\) for \(1\leq j\leq k\).
For an admissible class $r \bmod L(\mathbf{de})$, \(Q_j(r)/d_j=e_jQ_j(r)/(d_je_j)\), so
\[
 \overline{Q_j(r)/(d_je_j)}^{\,2}\bar e_j^{\,2}
 \equiv \overline{Q_j(r)/d_j}^{\,2}\,(\text{mod}\,d_j).
\]
The analogous
argument holds when \(\mathbf d\) and \(\mathbf e\) are interchanged.
Therefore, the twisted
multiplicativity formula \eqref{eq:Kl-mult} gives
\[
 \Kl\!\left(a\,\overline{Q_j(r)/(d_je_j)}^{\,2};d_je_j\right)
 =
 \Kl\!\left(a\,\overline{Q_j(r)/d_j}^{\,2};d_j\right)
 \Kl\!\left(a\,\overline{Q_j(r)/e_j}^{\,2};e_j\right).
\]
The factors on the right depend only on \(r\bmod L(\mathbf d)\)
and \(r\bmod L(\mathbf e)\), respectively. We can then conclude that \(H_{a,\bnu}(\mathbf d\mathbf e)=     
H_{a,\bnu}(\mathbf d)H_{a,\bnu}(\mathbf e)\)
if  \((d_1\cdots d_k, e_1\cdots e_k)=1\).

For \eqref{eq:H-comparison}, the bound \eqref{eq:Kl-Weil} implies
\[
\begin{aligned}
 \prod_{j=1}^k
\left|\Kl\!\left(a\,\overline{Q_j(r)/d_j}^{\,2};d_j\right)\right|^{\nu_j}
& \leq \prod_{j=1}^k \big(2^{\omega(d_j)}\kappa(v_2(d_j))\big)^{\nu_j},  
\end{aligned}
\]
and thus
\begin{align}\label{eq0}
\frac{ H_{a,\bnu}(\mathbf d)}{\prod_{j=1}^k \big(2^{\omega(d_j)}\kappa(v_2(d_j))\big)^{\nu_j}}
 &\le
 \frac{\#\{r\bmod L(\mathbf d):\mathcal P_{\mathbf Q}(\mathbf d,r)\}}
      {L(\mathbf d)}.
\end{align}
As \(L(\mathbf d)| (d_1\cdots d_k)^2\),  the densities of admissible residues classes modulo \(L(\mathbf d)\) and \((d_1\cdots d_k)^2\) are the same. For an admissible class $r$, we have \(d_j\mid Q_j(r)\) for \(1\leq j\leq k\), so \(d_1\cdots d_k\mid Q(r)\). Altogether, we infer that
\[
\begin{aligned}
 \frac{\#\{r\bmod L(\mathbf d):\mathcal P_{\mathbf Q}(\mathbf d,r)\}}
      {L(\mathbf d)}
 &=\frac{\#\{r\bmod(d_1\cdots d_k)^2:
                \mathcal P_{\mathbf Q}(\mathbf d,r)\}}
           {(d_1\cdots d_k)^2}\\
 &\leq\frac{\#\{r\bmod(d_1\cdots d_k)^2:
                   d_1\cdots d_k\mid Q(r)\}}
              {(d_1\cdots d_k)^2}\\
 &=\frac{\rho(d_1\cdots d_k)}{d_1\cdots d_k}.
\end{aligned}
\]
With  \eqref{eq0} and the definition of $\kappa(\cdot)$ (under  
\eqref{eq:Kl-Weil}), our assertion \eqref{eq:H-comparison} follows readily.
\end{proof}

Let \(R,S\in\Z[X]\) be two nonconstant polynomials. The following facts are well-known. First, the discriminant \(\disc (R)\) is nonzero if and only if all roots of \(R\) are simple. Second, the resultant \(\Res(R,S)\) is nonzero if and only if \(R\) and \(S\) have no common roots. Moreover, we have the standard identity
\begin{equation}\label{eq:discriminant-product}
  \disc(RS)=\disc(R)\disc(S)\Res(R,S)^2.
\end{equation}
Applying them to our case $Q= Q_1\cdots Q_k$,  for \(p\nmid D=\disc(Q)\), we get \(p\nmid\Res(Q_i,Q_j)\) for \(i\ne j\). Thus,
\begin{enumerate}
    \item \(Q_1,\ldots, Q_k\) have simple roots \(\bmod p\);
    \item any pair of \(Q_1,\ldots, Q_k\)  have no common root modulo \(p\).
\end{enumerate}
   This yields a strong constraint to the non-vanishing of \(H_{a,\bnu}(\mathbf d)\), as shown below.
\begin{lemma}\label{lem:local-majorant}
  Assume \(p\nmid D\) and  let \({\boldsymbol\ell}=(\ell_1,\ldots,\ell_k)\in \N_0^k\). 

  Case (1): 
    \(\ell_i,\ell_j>0\) for some $i\neq j$:  
  \[
  H_{a,\bnu}(p^{\ell_1},\ldots,p^{\ell_k})=0.
  \]
  Consequently, for \(\mathbf d\in\N^k\) satisfying
\((d_1\cdots d_k,D)=1\), \(H_{a,\bnu}(\mathbf d)\neq 0\) can happen only when \((d_i,d_j)=1\) for all $i\neq j$, and in this case, 
  \begin{equation}\label{eq:H-majorant-coprime-D}
    H_{a,\bnu}(\mathbf d) = \prod_{j=1}^k H_{a,\bnu}(1,\ldots,d_j,\ldots,1)=  \prod_{j=1}^k \prod_{p^\ell \| d_j} H_{a,\bnu}(1,\ldots,p^\ell,\ldots,1).
  \end{equation}

  Case (2): \(\ell_i=0\) for all \(i\neq j\): 
  \begin{equation}\label{eq:H-good}
    H_{a,\bnu}(1,\ldots,p^\ell,\ldots,1)
    = (1-p^{-1})\rho_{Q_j}(p)
    \frac{\lambda_{a,\nu_j}(p^\ell)}{p^\ell}
    \qquad \mbox{if \(p\nmid\lc(Q)\)D}.
  \end{equation}

\end{lemma}

\begin{proof}
Suppose \(\ell_i,\ell_j\ge 1\) for
  \(i\ne j\). Write  \(p^{\boldsymbol\ell} = (p^{\ell_1},\ldots,p^{\ell_k})\). The condition
  \(\mathcal P_{\mathbf Q}(p^{\boldsymbol\ell},r)\) would
  require \(p^{\ell_i}\mid Q_i(r)\) and \(p^{\ell_j}\mid Q_j(r)\), so \(Q_i\) and \(Q_j\) have a common root mod \(p\), as $p\nmid D$.  This contradicts (2) above. 
  Hence \(\mathcal P_{\mathbf Q}(p^{\boldsymbol\ell},r)\) does not hold for any \(r\),  implying \(H_{a,\bnu}(p^{\boldsymbol\ell})=0\). 
  
  Using the multiplicativity shown in Lemma~\ref{lem:H-properties},  \( H_{a,\bnu}(\mathbf d)\) decomposes into   
  \[H_{a,\bnu}(\mathbf d)= \prod_p H_{a,\bnu}(p^{{\boldsymbol\ell}_p})\]
  for some ${\boldsymbol\ell}_p\in \N_0^k$. If $(d_i,d_j)>1$ for some $i\neq j$, then some of $H_{a,\bnu}(p^{{\boldsymbol\ell}_p})$'s vanish.  The two equalities in \eqref{eq:H-majorant-coprime-D} follow from the multiplicativity, and we need to note the use of $(d_i,d_j)=1$ for all $i\neq j$ in the first equality.

  Suppose that \(\ell_j=\ell\geq1\) and \(\ell_i=0\) for \(i\ne j\).
  The condition \(\mathcal P_{\mathbf Q}(p^{\boldsymbol\ell},r)\)
  then reduces to \(p^\ell\parallel Q_j(r)\) and \(p\nmid Q_i(r)\) for \(i\ne j\).  Moreover, by (1), the roots of \(Q_j\) modulo \(p\) are simple, thus by Hensel's lemma, \(\rho_{Q_j}(p^\ell)=\rho_{Q_j}(p)\).
  When \(p\nmid\lc(Q)\) as well, 
  \(p\nmid\lc(Q_j)\disc(Q_j)\) and so  Lemma~\ref{lem:lifting} applies. We therefore infer that 
  \[
  \begin{aligned}
    H_{a,\bnu}(1,\ldots,p^\ell,\ldots,1)
    &=\frac1{p^{2\ell}}
      \sum_{\substack{r\bmod p^{2\ell}\\p^\ell\parallel Q_j(r)}}
      \left|\Kl\!\left(a\,\overline{Q_j(r)/p^\ell}^{\,2};p^\ell\right)\right|^{\nu_j}\\
    &=\rho_{Q_j}(p) \frac{\varphi(p^\ell)\lambda_{a,\nu_j}(p^\ell)}{p^{2\ell}}.
  \end{aligned}
  \]
  This proves \eqref{eq:H-good}.
\end{proof}

To control the contribution of prime powers in the following lemma,
we use Stewart's estimate \cite[Corollary~2 and (44)]{Stewart-1991}
\begin{equation}\label{eq:Stewart}
 \rho_R(p^\ell)
 \leq d\,p^{\lfloor(1-1/d)\ell\rfloor}
 \qquad(\ell\geq1),
\end{equation}
where \(d=\deg R\) and \(R\in\{Q,Q_1,\ldots,Q_k\}\).

\begin{lemma}\label{lem:T-Delta} 
Let \(T_{a,\bnu,F}(\mathbf d)=  G_F(\mathbf d)H_{a,\bnu}(\mathbf d)\) and \(\Delta_p(a,\bnu,F)\) be defined as in \eqref{eq:T-Delta-def}. 
\begin{enumerate}
\item  \(T_{a,\nu,F}\) is submultiplicative, and  
\begin{align}\label{eq:moderate}
 \sum_{|\boldsymbol\ell|\geq s}T_{a,\nu,F}(p^{\boldsymbol\ell})
 &\ll_{g,k,A,B} \eta_2(p)^{|\bnu|} p^{-\max(1,s/(4g))}\nonumber\\ &\ll_{\bnu,g,k,A,B} p^{-\max(1,s/(4g))},
\end{align}
where \(\eta_2(p)= 2^{5/2}\) if \(p=2\), or \(2\) otherwise. 
\item 
\[\Delta_p(a,\bnu,F)  = \left(1+O(p^{-2})\right)
 \Big(1+\sum_{1\leq|\boldsymbol\ell|<8g}
 T_{a,\nu,F}(p^{\boldsymbol\ell})\Big) 
\]
where the $O$-term denotes a nonnegative quantity and the implied $O$-constant depends only on \(\bnu,g,k,A,B\). 
\end{enumerate}
Hence  
    \begin{equation}\label{eq:Delta}
        1\le \Delta_p(a,\bnu,F) \le  1+Cp^{-1}
    \end{equation}     
for some positive constant $C$ depending only on  \(\bnu,g,k,A,B\). 
\end{lemma}
\begin{proof}
    Let \(T(\mathbf d):=T_{a,\nu,F}(\mathbf d)\).
The facts that  \(T(\mathbf1_k)=1\) and 
\[ T(\mathbf a\mathbf b)\leq T(\mathbf a)T(\mathbf b) \quad \mbox{ for \((a_1\cdots a_k,b_1\cdots b_k)=1\)}
\] 
follow from  \eqref{submultiplicativity} and the multiplicativity of  \(H_{a,\bnu}\).
This justifies the submultiplicativity of \(T\).

For \(|\boldsymbol\ell|=u\geq1\),  \eqref{eq:standard-majorant-bound} and \eqref{eq:H-fair-2}
 imply
\[
 T(p^{\boldsymbol\ell})
 \ll \eta_2(p)^{|\bnu|}\min(A^u,Bp^{\varepsilon u})\frac{\rho(p^u)}{p^u}.
\]
Using \(\rho(p^u)\leq p^{u-1}\rho(p)\leq gp^{u-1}\) and
Stewart's estimate \eqref{eq:Stewart}, we get
\begin{align}\label{eq:se}
 T(p^{\boldsymbol\ell})
 \ll_{g} \eta_2(p)^{|\bnu|}
 \min(A^u,Bp^{\varepsilon u})\min(p^{-1},p^{-u/g}).    
\end{align}
There are \(\binom{u+k-1}{k-1}\) vectors
\(\boldsymbol\ell\) with \(|\boldsymbol\ell|=u\).
For \(1\leq u<4g\), the factor \(A^u\) is bounded in terms of
\(A,g\). For \(u\geq4g\), the assumption
\(\varepsilon\leq1/(2g)\) and the bound
\(\binom{u+k-1}{k-1}2^{-u/(4g)}\ll_{g,k}1\) imply
\[
 \binom{u+k-1}{k-1}Bp^{-(1/g-\varepsilon)u}
 \ll_{g,k,B}p^{-u/(4g)}.
\]
Summation over \(u\geq s\) therefore proves
\begin{equation*}
 \sum_{|\boldsymbol\ell|\geq s}T(p^{\boldsymbol\ell})
 \ll_{g,k,A,B}\eta_2(p)^{|\bnu|} p^{-\max(1,s/(4g))}.
\end{equation*}

As  
\[
\Delta_p(a,\bnu,F) =   1+\sum_{|\boldsymbol\ell|\geq1}T(p^{\boldsymbol\ell}),
\]
we take \(s=8g\) and get from the nonnegativity of \(T\) that
\[
\begin{aligned}
 \Delta_p(a,\bnu,F) &= 1+ \sum_{1\leq|\boldsymbol\ell|<8g}
 T(p^{\boldsymbol\ell}) + O(p^{-2})\\
 & = \left(1+O(p^{-2})\right)
 \Big(1+\sum_{1\leq|\boldsymbol\ell|<8g}
 T(p^{\boldsymbol\ell})\Big)
\end{aligned} 
\]
with nonnegative $O$-terms, where the implied $O$-constants depend on \(\bnu,g,k,A,B\) only. 

We apply \eqref{eq:se} to $T(p^{\boldsymbol\ell})$ for $|\boldsymbol\ell|< 8g$, thus
\[
1\le  \Delta_p(a,\bnu,F)  = \left(1+O(p^{-2})\right)\left(1+O(p^{-1})\right)\le 1+ C p^{-1}
\]
where the constant $C>0$ depends only on \(\bnu,g,k,A,B\).
\end{proof}

\section{Sieve and stability lemmas}

For the sieve estimates below and the proof of
Theorem~\ref{thm:main}, we fix \(0<\alpha<1\) and set
\begin{equation}\label{eq:epsilons}
 \varepsilon_1=\frac{3\alpha}{25},\qquad
 \varepsilon_2=\frac{\alpha}{25},\qquad
 \varepsilon_3=\frac{\alpha}{50g}.
\end{equation}

We begin with a sieve estimate for polynomial values along arithmetic
progressions.

\begin{lemma}\label{lem:polynomial-sieve}
	Let \(Q\in\Z[X]\) be primitive of degree \(g\geq1\), let
	\(M\in\N\), \(r\in\Z\), and suppose that
	\[
	 x^\alpha<y\leq x,\qquad
	 M\leq x^{2\varepsilon_1},\qquad
	 2\leq z\leq x^{\varepsilon_3}.
	\]
	Then we have
	\[
		\sum_{\substack{\ x<n\leq x+y\\
				n\equiv r\,({\rm mod}\,M)\\
				p\mid Q(n),\ p\nmid M\ \Longrightarrow\ p>z}}
		1
		\ll
		\frac yM\prod_{\substack{g<p\leq z\\p\nmid M}}
		\left(1-\frac{\rho_Q(p)}p\right),
	\]
	where the implied constant depends only on \(g\).
\end{lemma}

\begin{proof}
Write \(\rho=\rho_Q\) and \(n=Mm+r\). The interval for \(m\) is
\[
 I=\left(\frac{x-r}{M},\frac{x+y-r}{M}\right],\qquad
 X=|I|=\frac yM>x^{\alpha-2\varepsilon_1}=x^{19\alpha/25}.
\]
The case \(z\leq g\) follows from
\(\#(I\cap\Z)\leq X+1\). We may thus suppose that \(z>g\), and define
\[
 \mathfrak P_g=\prod_{\substack{g<p\leq z\\p\nmid M}}p.
\]
The sum in the lemma does not exceed
\[
 \#\{m\in I\cap\Z:(Q(Mm+r),\mathfrak P_g)=1\}.
\]

For \(d\mid\mathfrak P_g\), the congruence \(Q(Mm+r)\equiv0\, (\text{mod}\,d)\) has \(\rho(d)\) solutions modulo \(d\), since \((d,M)=1\). Hence
\begin{equation}\label{eq:sieve-remainder}
 \#\{m\in I\cap\Z:d\mid Q(Mm+r)\}
 =X\frac{\rho(d)}d+R_d,
 \qquad |R_d|\leq\rho(d).
\end{equation}

Since \(\rho(p)\leq g\), we have 
\[
 \left(1-\frac{\rho(p)}p\right)^{-1}
 \leq\left(1-\frac1p\right)^{-g}
       \bigl(1+O_g(p^{-2})\bigr)
       \quad \mbox{for \(p>2g\)},
\]
or  \(\le g+1\) for \(g<p\leq2g\).
It follows by Mertens' theorem that
\begin{equation}\label{eq:sieve-dimension-axiom}
 \prod_{\substack{w\leq p<v\\p>g,\ p\nmid M}}
 \left(1-\frac{\rho(p)}p\right)^{-1}
 \ll_g\left(\frac{\log v}{\log w}\right)^g
 \qquad(2\leq w<v),
\end{equation}

We take \(D_0=z^{20g}\). Since \(\log D_0/\log z=20g>9g+1\),
the upper bound in
\cite[Corollary~6.2]{IK-2004}, together with
\eqref{eq:sieve-remainder} and \eqref{eq:sieve-dimension-axiom}, yields
\[
 \#\{m\in I\cap\Z:(Q(Mm+r),\mathfrak P_g)=1\}
 \ll_g X  \prod_{\substack{g<p\leq z\\p\nmid M}}
 \left(1-\frac{\rho(p)}p\right)
 +\sum_{\substack{d\leq D_0\\d\mid\mathfrak P_g}} \rho(d).
\]
Since \(d\mid\mathfrak P_g\) is squarefree, we have
\(\rho(d)=\prod_{p\mid d}\rho(p)\leq g^{\omega(d)}\). Hence, we get
\[
 \sum_{\substack{d\leq D_0\\d\mid\mathfrak P_g}}\rho(d)
 \leq\sum_{d\leq D_0}g^{\omega(d)}
 \ll_g D_0(\log D_0)^{g-1}.
\]
Since \(X>x^{19\alpha/25}\geq z^{38g}\) and \(D_0=z^{20g}\),
it follows from \eqref{eq:sieve-dimension-axiom} that
\[
 X\prod_{\substack{g<p\leq z\\p\nmid M}}
 \left(1-\frac{\rho(p)}p\right)
 \gg_g X(\log z)^{-g}
 \gg_g D_0(\log D_0)^{g-1}.
\]
This proves the lemma.
\end{proof}

The next lemma is an analogue of the upper bound in
\cite[Lemma~6]{Henriot-2012} with Kloosterman weights under the
corrected congruence condition of
\cite[pp.~375--376]{Henriot-2014-erratum}.

\begin{lemma}\label{lem:weighted-sieve}
Let \(x^\alpha<y\leq x\) and \(z=x^{\varepsilon_3}\).
Then, uniformly for \(\mathbf d\in\N^k\) with
\(d_1\cdots d_k\leq x^{\varepsilon_1}\), we have
\[
\sum_{\substack{x<n\leq x+y\\
                  \mathcal P_{\mathbf Q}(\mathbf d,n)\\
                  p\mid Q(n),\ p\nmid d_1\cdots d_k\Rightarrow p>z}}
 \prod_{j=1}^k
 \left|\Kl\!\left(
 a\,\overline{Q_j(n)/d_j}^{\,2};d_j\right)\right|^{\nu_j}\ll yH_{a,\bnu}(\mathbf d)
 \prod_{\substack{g<p\leq z\\p\nmid d_1\cdots d_k}}
 \left(1-\frac{\rho(p)}p\right),
\]
where the implied constant depends only on \(g\).
\end{lemma}

\begin{proof}
Write \(L=L(\mathbf d)\). 
By \eqref{key} and \eqref{p-property}, we may write the sum on the left-side into
\begin{equation}\label{eq:sum-rm}
 \sum_{\substack{0\leq r<L\\
                  \mathcal P_{\mathbf Q}(\mathbf d,r)}}
 \prod_{j=1}^k
 \left|\Kl\!\left(
 a\,\overline{Q_j(r)/d_j}^{\,2};d_j
 \right)\right|^{\nu_j}
 \sum_{\substack{x<n\leq x+y\\
                  n\equiv r\,(\text{mod}\,L)\\
                  p\mid Q(n),\ p\nmid L\Rightarrow p>z}}1.
\end{equation}

Since \(L\leq(d_1\cdots d_k)^2\leq x^{2\varepsilon_1}\),
Lemma~\ref{lem:polynomial-sieve} applies with \(M=L\). Thus the inner sum  in \eqref{eq:sum-rm} is
\[
 \ll_g\frac yL
 \prod_{\substack{g<p\leq z\\p\nmid L}}
 \left(1-\frac{\rho(p)}p\right).
\]
Inserting this bound in \eqref{eq:sum-rm}, we obtain
the desired estimate by the definition of \(H_{a,\bnu}(\mathbf d)\).
\end{proof}

\begin{lemma}\label{lem:weighted-sieve-ell}
Let \(x^\alpha<y\leq x\). Suppose that \(\mathbf d\in\N^k\),
\(r\in\N\), and a prime \(q\nmid d_1\cdots d_k\) satisfy
\[
 q\leq x^{\varepsilon_3},\qquad
 d_1\cdots d_k\,q^r\leq x^{\varepsilon_1}.
\]
Then we have
\[
\sum_{\substack{\ x<n\leq x+y\\
                  \mathcal P_{\mathbf Q}(\mathbf d,n)\\
                  q^r\mid Q(n)\\
                  p\mid Q(n),\
                  p\nmid qd_1\cdots d_k\Rightarrow p>q}}
 \prod_{j=1}^k
 \left|\Kl \left(
 a\,\overline{Q_j(n)/d_j}^{\,2};d_j\right)\right|^{\nu_j}\ll y\frac{\rho(q^r)}{q^r}H_{a,\bnu}(\mathbf d)
 \prod_{\substack{g<p\leq q\\p\nmid qd_1\cdots d_k}}
 \left(1-\frac{\rho(p)}p\right),
\]
where the implied constant depends only on \(g\).
\end{lemma}

\begin{proof}
Write \(L=L(\mathbf d)\) and \(M=Lq^r\). By \eqref{key} and \eqref{p-property}, the condition
\(\mathcal P_{\mathbf Q}(\mathbf d,n)\) and the product of absolute
powers of Kloosterman sums depend only on \(n\bmod L\).
Since \(\rad L=\rad(d_1\cdots d_k)\) and \(\rad M=\rad(qd_1\cdots d_k)\), we decompose the sum into
\[
 \sum_{\substack{t\bmod M\\
                  \mathcal P_{\mathbf Q}(\mathbf d,t)\\
                  q^r\mid Q(t)}}
 \prod_{j=1}^k
 \left|\Kl\left(a\,\overline{Q_j(t)/d_j}^{\,2};d_j\right)\right|^{\nu_j}
 \sum_{\substack{x<n\leq x+y\\
                  n\equiv t\,(\text{mod}\,M)\\
                  p\mid Q(n),\ p\nmid M\Rightarrow p>q}}1.
\]

Since \(M\leq(d_1\cdots d_k\,q^r)^2\leq x^{2\varepsilon_1}\),
Lemma~\ref{lem:polynomial-sieve} with \(z=q\) shows that the inner sum is
\[
 \ll_g\frac yM
 \prod_{\substack{g<p\leq q\\p\nmid M}}
 \left(1-\frac{\rho(p)}p\right).
\]
Fix \(u\bmod L\) for which \(\mathcal P_{\mathbf Q}(\mathbf d,u)\) holds.
Since \((L,q)=1\) and \(M=q^rL\), the Chinese remainder theorem yields that there are exactly
\(\rho(q^r)\) residue classes \(t\bmod M\) satisfying
\(t\equiv u\,(\text{mod}\,L)\) and \(q^r\mid Q(t)\). Moreover, \eqref{p-property} holds for $(r,r')=(t,n)$, thus
\begin{align*}
 \frac1M
 \sum_{\substack{t\bmod M\\
                  \mathcal P_{\mathbf Q}(\mathbf d,t)\\
                  q^r\mid Q(t)}}
 \prod_{j=1}^k
 \left|\Kl\left(a\,\overline{Q_j(t)/d_j}^{\,2};d_j\right)\right|^{\nu_j}
 &\qquad=\frac{\rho(q^r)}{Lq^r}
 \sum_{\substack{u\bmod L\\
                  \mathcal P_{\mathbf Q}(\mathbf d,u)}}
 \prod_{j=1}^k
 \left|\Kl\left(a\,\overline{Q_j(u)/d_j}^{\,2};d_j\right)\right|^{\nu_j}
 \\
 &\qquad=\frac{\rho(q^r)}{q^r}H_{a,\bnu}(\mathbf d).
\end{align*}
This proves the lemma.
\end{proof}

\begin{lemma}\label{lem:stability}
Let \(F\in\mathcal M_k(A,B,\varepsilon)\), where
\(0<\varepsilon\leq1/(2g)\), and define
\[
 W(\mathbf d)=F(\mathbf d)H_{a,\bnu}(\mathbf d).
\]
Then the following estimates hold.
\begin{enumerate}[label=(\roman*)]
\item If \(\theta:\N\to\R_{\geq0}\) is multiplicative and
      \(\theta(p^\ell)=1+O(p^{-1})\) uniformly in \(p\) and
      \(\ell\geq1\), then
      \[
      \sum_{d_1\cdots d_k\leq z}W(\mathbf d)\theta(d_1\cdots d_k)
      \ll\sum_{d_1\cdots d_k\leq z}W(\mathbf d)
      \qquad(z\geq1).
      \]
      The same estimate holds with both summation ranges replaced by
      \(P^+(d_1\cdots d_k)\leq z\).
\item If \(\theta\) is as in \textup{(i)}, \(\chi>0\),
      \(z\geq\e^{8g\chi}\), and \(\sigma=\chi/\log z\), then
      \[
      \sum_{P^+(d_1\cdots d_k)\leq z}
      W(\mathbf d)\theta(d_1\cdots d_k)(d_1\cdots d_k)^\sigma
      \ll\sum_{P^+(d_1\cdots d_k)\leq z}W(\mathbf d).
      \]
\item For \(z\geq1\),
      \[
      \sum_{P^+(d_1\cdots d_k)\leq z}W(\mathbf d)
      \asymp\sum_{d_1\cdots d_k\leq z}W(\mathbf d).
      \]
\end{enumerate}
The implied constants depend only on \(\bnu,g,k,A,B\), and in
\textup{(i)} and \textup{(ii)} also on the uniform constant in the
hypothesis on \(\theta\). The constant in \textup{(ii)} may
additionally depend on \(\chi\).
\end{lemma}

\begin{proof}
Let \(T(\mathbf d)=G_F(\mathbf d)H_{a,\bnu}(\mathbf d)\).
Suppose that \((a_1\cdots a_k,b_1\cdots b_k)=1\). By 
\eqref{submultiplicativity} and the multiplicativity of  \(H_{a,\bnu}\), we have 
\begin{equation}\label{eq:weighted-submultiplicativity}
 W(\mathbf a\mathbf b)\leq T(\mathbf a)W(\mathbf b).
\end{equation}
We also have \(W\leq T\) and \(W(\mathbf1_k)=T(\mathbf1_k)=1\).

Our assertion (i) is a close analogue of \cite[Lemma~2]{Henriot-2012}, whose proof and its correction in \cite[p.~376]{Henriot-2014-erratum} rely only on inequalities as in 
\eqref{eq:weighted-submultiplicativity} and \eqref{eq:moderate}.
Therefore we can prove \textup{(i)} using the argument in \cite{Henriot-2012}. 
The summations with ranges over  \(P^+(d_1\cdots d_k)\leq z\) remain convergent because
\[
 \sum_{P^+(d_1\cdots d_k)\leq z}W(\mathbf d)
 \leq\prod_{p\leq z}\bigg(1+
       \sum_{\boldsymbol\ell\ne\mathbf0}T(p^{\boldsymbol\ell})\bigg)
 <\infty
\]
by \eqref{eq:moderate}.   
 These summations supported on  
\(P^+(d_1\cdots d_k)\leq z\) can be handled in the same fashion since they are dominated by the same Euler products over \(p\leq z\).
By Lemma~\ref{lem:T-Delta} (1), we see that 
 the sums over integers with bounded prime factors
converge,

For \textup{(ii)}, the inequalities in \eqref{eq:weighted-submultiplicativity}
remain valid after multiplying \((d_1\cdots d_k)^\sigma\) to \(W(\mathbf d)\) and \(T(\mathbf d)\). Since \(\theta(p^\ell)\ll1\)
uniformly for \(\ell\geq1\), we also have
\[
 \sum_{|\boldsymbol\ell|\geq s}
 T(p^{\boldsymbol\ell})\theta(p^{|\boldsymbol\ell|})
 \ll p^{-\max(1,s/(4g))}
 \qquad(s\geq1).
\]
To apply the corrected proof of \cite[Lemma~3]{Henriot-2012},
we note that the hypothesis on \(z\) ensures
\(\sigma\leq1/(8g)\), and hence
\[
 \sum_{s>8g}p^{(\sigma-1/(4g))s}
 \ll_g p^{-1-1/(8g)}.
\]
Thus the contribution from \(s>8g\) is summable over the primes.
The corrected argument and \textup{(i)} therefore yield
\[
\begin{aligned}
 &\sum_{P^+(d_1\cdots d_k)\leq z}
 W(\mathbf d)\theta(d_1\cdots d_k)(d_1\cdots d_k)^\sigma\\
 &\qquad\ll\sum_{P^+(d_1\cdots d_k)\leq z}
 W(\mathbf d)\theta(d_1\cdots d_k)
 \ll\sum_{P^+(d_1\cdots d_k)\leq z}W(\mathbf d).
\end{aligned}
\]

Finally, we apply the argument of \cite[Lemma~5]{Henriot-2012} and \textup{(ii)} with \(\theta=1\), \(\chi=1\) and the normalization \(W(\mathbf1_k)=1\) to derive
\textup{(iii)}, based on the same spirit that the counterparts of \eqref{eq:weighted-submultiplicativity} and \eqref{eq:moderate} are the only concerned ingredients for the proof of \cite[Lemma~4]{Henriot-2012}.

The dependency of the implied constants is clear from the arguments of our proof. 
\end{proof}

\vskip 5mm
\section{Proof of Theorem \ref{thm:main}}

We follow the decomposition of Nair and Tenenbaum
\cite[Section~4]{NT-1998} and Henriot
\cite[Section~5]{Henriot-2012}.
All implied constants in this section depend at most on
\(\bnu,g,k,\alpha,\delta,A,B\).
We may assume that
\[
 \varepsilon=\frac{\alpha}{100g(g+\delta^{-1})}.
\]
Indeed, \(M_k(A,B,\varepsilon') \subset M_k(A,B,\varepsilon)\) if $0< \varepsilon'<\varepsilon$. With this choice, the parameters in \eqref{eq:epsilons} satisfy
\begin{equation}\label{eq:epsilon-budget}
 2(g+\delta^{-1})\varepsilon
 =\varepsilon_3=\frac{\varepsilon_2}{2g},
 \qquad \varepsilon_1=3\varepsilon_2.
\end{equation}
In particular, \(\varepsilon\leq1/(2g)\), as required in
Lemma~\ref{lem:stability}.
Let \(z=x^{\varepsilon_3}\), and choose \(c_0\) so that \(2^gc_0^{-1/\delta}\leq1\).
For \(c_0\le x<n\leq x+y\leq2x\), we then have
\begin{equation}\label{eq:Q-size-uniform}
 |Q(n)|\leq\|Q\|(2x)^g
 \leq2^gc_0^{-1/\delta}x^{g+\delta^{-1}}
 \leq x^{g+\delta^{-1}}.
\end{equation}

For \(x<n\leq x+y\) with \(Q(n)\neq0\), write
\(|Q(n)|=d_nb_n\), where \(d_n\) is maximal subject to
\[
 d_n\leq x^{\varepsilon_1},\qquad
 P^+(d_n)<P^-(b_n).
\]
Define 
\[
 d_{jn}=\prod_{p\mid d_n}p^{v_p(Q_j(n))},
 \qquad b_{jn}=\frac{|Q_j(n)|}{d_{jn}}
 \qquad(1\leq j\leq k).
\]
Then \(d_n=d_{1n}\cdots d_{kn}\),
\(b_n=b_{1n}\cdots b_{kn}\), \((d_n,b_n)=1\), and
\(\mathcal P_{\mathbf Q}(\mathbf d_n,n)\) holds where \(\mathbf d_n =(d_{1n},\ldots,d_{kn})\).
Let \(\eta_0>2g\) be a constant to be chosen below. We divide the
integers \(n\) into the five classes
\[
\begin{aligned}
 (C_1) \quad& b_n>1,\quad P^-(b_n)>z;\\
 (C_2)\quad& d_n\leq x^{\varepsilon_2},\quad P^-(b_n)\leq z;\\
 (C_3)\quad& d_n>x^{\varepsilon_2},\quad
             P^+(d_n)>\eta_0,\quad P^-(b_n)\leq z;\\
 (C_4)\quad& d_n>x^{\varepsilon_2},\quad
             P^+(d_n)\leq\eta_0,\quad P^-(b_n)\leq z;\\
 (C_5)\quad& b_n=1.
\end{aligned}
\]
Let \(S_i\) denote the contribution of \(C_i\) to the left-hand side of
\eqref{eq:main-theorem}.

\vskip 2mm

\subsection{Estimation of $S_1$}
For \(n\in C_1\), we have \(b_n\geq P^-(b_n)^{\Om(b_n)}>z^{\Om(b_n)}\). The estimate \eqref{eq:Q-size-uniform} gives
\[
 \Om(b_n)\leq\frac{\log b_n}{\log z} \leq\frac{\log|Q(n)|}{\log z}
 \leq\frac{g+\delta^{-1}}{\varepsilon_3}.
\]
Since \((d_n,b_n)=1\), it follows from \eqref{eq:Mclass} that
\[
 F(|Q_1(n)|,\ldots,|Q_k(n)|)
 \leq A^{\Om(b_n)}F(\mathbf d_n)
 \ll F(\mathbf d_n).
\]
For each \(j\), the factorization \(|Q_j(n)|=d_{jn}b_{jn}\)
and the twisted multiplicativity \eqref{eq:Kl-mult} yield
\[
 \Kl(a;|Q_j(n)|)
 =\Kl(a\bar b_{jn}^{\,2};d_{jn})
  \Kl(a\bar d_{jn}^{\,2};b_{jn}).
\]
Applying the bound \eqref{eq:Kl-Weil}, we get
\[
 \prod_{j=1}^k
 |\Kl(a\bar d_{jn}^{\,2};b_{jn})|^{\nu_j}
 \leq 2^{\frac32|\bnu|+\sum_{j=1}^k\nu_j\omega(b_{jn})}
 \leq 2^{(\frac32+\Om(b_n))|\bnu|}
 \ll1.
\]
Combining these estimates, we obtain
\[
F(|Q_1(n)|,\ldots,|Q_k(n)|)
 \prod_{j=1}^k|\Kl(a;|Q_j(n)|)|^{\nu_j}\ll F(\mathbf d_n)
 \prod_{j=1}^k
 \left|\Kl\left(a\,\overline{Q_j(n)/d_{jn}}^{\,2};d_{jn}\right)
 \right|^{\nu_j},
\]
which implies 
\[
 S_1\ll
 \sum_{d_1\cdots d_k\leq x^{\varepsilon_1}}F(\mathbf d)
 \sum_{\substack{x<n\leq x+y\\
                  \mathcal P_{\mathbf Q}(\mathbf d,n)\\
                  p\mid Q(n),\ p\nmid d_1\cdots d_k\Rightarrow p>z}}
 \prod_{j=1}^k
 \left|\Kl\left(a\,\overline{Q_j(n)/d_j}^{\,2};d_j\right)
 \right|^{\nu_j}.
\]
We apply Lemma~\ref{lem:weighted-sieve} to the inner sum and obtain
\[
 S_1\ll y\sum_{d_1\cdots d_k\leq x^{\varepsilon_1}}W(\mathbf d)
 \prod_{\substack{g<p\leq z\\p\nmid d_1\cdots d_k}}
 \left(1-\frac{\rho(p)}p\right).
\]
By \eqref{eq:sieve-dimension-axiom}, the inner product satisfies 
\begin{equation}\label{eq:finiteprod-pho}
	\begin{aligned}
		\prod_{\substack{g<p\le z\\p\nmid d_1\cdots d_k}}
		\left(1-\frac{\rho(p)}p\right)
		&=V_g(x)
		\prod_{\substack{z<p\le x\\p>g}}
		\left(1-\frac{\rho(p)}p\right)^{-1}
		\prod_{\substack{g<p\le z\\p\mid d_1\cdots d_k}}
		\left(1-\frac{\rho(p)}p\right)^{-1}\\
		&\ll V_g(x)
		\left(\frac{\log x}{\log z}\right)^g\theta(d_1\cdots d_k)\\
		&\ll V_g(x)\theta(d_1\cdots d_k),
	\end{aligned}
\end{equation}
where \(\theta(m)=(m/\varphi(m))^g\) and 
\begin{equation}\label{eq:Vg}
 V_g(t):=\prod_{g<p\leq t}\left(1-\frac{\rho(p)}p\right).   
\end{equation} 
We deduce that 
\[
 S_1\ll y V_g(x)\sum_{d_1\cdots d_k\leq x}W(\mathbf d)\theta(d_1\cdots d_k).
\]
Since $\theta(m)$ is multiplicative and \(\theta(p^\ell)=1+O_g(p^{-1})\)
uniformly for \(\ell\geq1\), applying the assertion (i) in Lemma~\ref{lem:stability} gives
\begin{equation}\label{eq:S1}
 S_1\ll yV_g(x)\sum_{d_1\cdots d_k\leq x}W(\mathbf d).
\end{equation}

\subsection{Estimation of $S_3$}
For \(n\in C_3\), let \(\ell_n=P^+(d_n)\) and write $d_{jn}=\ell_n^{r_{jn}}c_{jn}$ with $\ell_n\nmid c_{jn}\, (1\leq j\leq k)$. We denote $r_n=r_{1n}+\cdots+r_{kn}$.
Thus
\[
 \frac{x^{\varepsilon_2}}{\ell_n^{r_n}}
 <c_{1n}\cdots c_{kn}
 \leq\frac{x^{\varepsilon_1}}{\ell_n^{r_n}},
 \qquad P^+(c_{1n}\cdots c_{kn})<\ell_n\leq z.
\]
Since \(P^-(b_n)>\ell_n\), we obtain from \eqref{eq:Q-size-uniform} that
\[
 \Om(b_n)\leq\frac{\log b_n}{\log\ell_n}
 \leq(g+\delta^{-1})\frac{\log x}{\log\ell_n}.
\]
The integers \(c_{1n}\cdots c_{kn}\), \(\ell_n^{r_n}\), and
\(b_n\) are pairwise coprime. Applying \eqref{eq:Mclass} twice therefore yields
\[
 F(|Q_1(n)|,\ldots,|Q_k(n)|)
 \leq A^{\Om(b_n)}\min(A^{r_n},B\ell_n^{\varepsilon r_n})
 F(\mathbf c_n).
\]
The twisted multiplicativity \eqref{eq:Kl-mult} gives
\[
 \Kl(a;|Q_j(n)|)
 =\Kl\left(a\,\overline{\ell_n^{r_{jn}}b_{jn}}^{\,2};c_{jn}\right)
  \Kl(a\bar c_{jn}^{\,2};\ell_n^{r_{jn}}b_{jn}).
\]
Since \(\omega(\ell_n^{r_{jn}}b_{jn})\leq1+\omega(b_{jn})\),
the bound \eqref{eq:Kl-Weil} implies
\[
 \prod_{j=1}^k
 |\Kl(a\bar c_{jn}^{\,2};\ell_n^{r_{jn}}b_{jn})|^{\nu_j}
 \leq2^{\frac52|\bnu|+\sum_j\nu_j\omega(b_{jn})}
 \ll2^{|\bnu|\Om(b_n)}.
\]
The bound for \(\Om(b_n)\) also implies
\[
 (2^{|\bnu|}A)^{\Om(b_n)}
 \leq(2^{|\bnu|}A)^{(g+\delta^{-1})\log x/\log\ell_n}.
\]
Moreover, \(\mathcal P_{\mathbf Q}(\mathbf c_n,n)\) holds and
\(\ell_n^{r_n}\mid Q(n)\). If \(p\mid Q(n)\) and
\(p\nmid\ell_nc_{1n}\cdots c_{kn}\), then \(p\mid b_n\) and
\(p>\ell_n\).

We group the terms by
\((\ell_n,r_n,\mathbf c_n)=(\ell,r,\mathbf c)\). With \(u=\log x/\log\ell\), we obtain
\[
\begin{split}
 S_3\ll{}&
 \sum_{\eta_0<\ell\leq z}\sum_{r\geq1}
 \min(A^r,B\ell^{\varepsilon r})
 (2^{|\bnu|}A)^{(g+\delta^{-1})u}\\
 &{}\times
 \sum_{\substack{x^{\varepsilon_2}<\ell^rc_1\cdots c_k
                         \leq x^{\varepsilon_1}\\
                  P^+(c_1\cdots c_k)<\ell}}F(\mathbf c)
 \sum_{\substack{x<n\leq x+y\\
                  \mathcal P_{\mathbf Q}(\mathbf c,n)\\
                  \ell^r\mid Q(n)\\
                  p\mid Q(n),\ p\nmid\ell c_1\cdots c_k\Rightarrow p>\ell}}
 \prod_{j=1}^k
 \left|\Kl\left(a\,\overline{Q_j(n)/c_j}^{\,2};c_j\right)
 \right|^{\nu_j}.
\end{split}
\]
Note that the sum over \(\ell\) is over primes. Employing Lemma~\ref{lem:weighted-sieve-ell} to bound the inner sum, we arrive at
\[
	S_3\ll y
	\sum_{\eta_0<\ell\leq x}\sum_{r\geq1}
	\min(A^r,B\ell^{\varepsilon r})
	\frac{\rho(\ell^r)}{\ell^r}
	(2^{|\bnu|}A)^{(g+\delta^{-1})u}
	\sum_{\substack{c_1\cdots c_k>x^{\varepsilon_2}/\ell^r\\
			P^+(c_1\cdots c_k)<\ell}}
	W(\mathbf c)
	\prod_{\substack{g<p\leq\ell\\p\nmid\ell c_1\cdots c_k}}
	\left(1-\frac{\rho(p)}p\right),
\]
where we have extended the sum over \(\ell\) to \(\ell\leq x\) and removed
the upper bound on \(c_1\cdots c_k\). As in the estimate of \(S_1\), we obtain from \eqref{eq:finiteprod-pho} that
\[
 \prod_{\substack{g<p\leq\ell\\p\nmid\ell c_1\cdots c_k}}
 \left(1-\frac{\rho(p)}p\right)
 \ll V_g(x)u^g\theta(\ell c_1\cdots c_k)
 \ll V_g(x)u^g\theta(c_1\cdots c_k),
\]
which implies 
\begin{equation}\label{eq:S3-1}
	\begin{aligned}
			S_3\ll& yV_g(x) 
		\sum_{\eta_0<\ell\leq x}\sum_{r\geq1}
		\min(A^r,B\ell^{\varepsilon r})
		\frac{\rho(\ell^r)}{\ell^r}
		(2^{|\bnu|}A)^{(g+\delta^{-1})u}u^g\\
		&\times
		\sum_{\substack{c_1\cdots c_k>x^{\varepsilon_2}/\ell^r\\
				P^+(c_1\cdots c_k)<\ell}}
		W(\mathbf c)\theta(c_1\cdots c_k).
	\end{aligned}
\end{equation}

We now choose $\chi=\frac{2+(g+\delta^{-1})\log(2^{|\bnu|}A)}
{\varepsilon_2}$, and set $\eta_0=\e^{8g\chi}$
and \(\sigma=\chi/\log\ell\). By Rankin's trick and the assertions (ii) and (iii) in  Lemma~\ref{lem:stability}, we derive
\[
\begin{split}
 \sum_{\substack{c_1\cdots c_k>x^{\varepsilon_2}/\ell^r\\
                  P^+(c_1\cdots c_k)<\ell}}
 W(\mathbf c)\theta(c_1\cdots c_k)&\leq
 \left(\frac{x^{\varepsilon_2}}{\ell^r}\right)^{-\sigma}
 \sum_{P^+(c_1\cdots c_k)\leq\ell}
 W(\mathbf c)\theta(c_1\cdots c_k)(c_1\cdots c_k)^\sigma\\
 &\quad\ll \e^{-\chi\varepsilon_2u+r\chi}
 \sum_{P^+(c_1\cdots c_k)\leq\ell}W(\mathbf c)\\
 &\quad\ll \e^{-\chi\varepsilon_2u+r\chi}
 \sum_{c_1\cdots c_k\leq x}W(\mathbf c).
\end{split}
\]
Moreover, by the choice of \(\chi\), we get
\[
 (2^{|\bnu|}A)^{(g+\delta^{-1})u}u^g \e^{-\chi\varepsilon_2u+r\chi}
 =\e^{-2u}u^g \e^{r\chi} \ll\e^{-u} \e^{r\chi}.
\]
Substituting these estimates into \eqref{eq:S3-1}, we find
\[
\begin{split}
 S_3\ll{}& yV_g(x)
 \left(\sum_{\eta_0<\ell\leq x}\e^{-u}
 \sum_{r\geq1}\e^{r\chi}\min(A^r,B\ell^{\varepsilon r})
 \frac{\rho(\ell^r)}{\ell^r}\right)
 \sum_{c_1\cdots c_k\leq x}W(\mathbf c).
\end{split}
\]

It follows from the bound \(\rho(\ell^r)\leq g\ell^{r-1}\) (see \eqref{eq:Stewart}) that
\[
 \sum_{r=1}^{2g}\e^{r\chi}\min(A^r,B\ell^{\varepsilon r})
 \frac{\rho(\ell^r)}{\ell^r}
 \ll\frac1\ell.
\]
For \(r>2g\), we  get
\[
 \sum_{r>2g}\e^{r\chi}\min(A^r,B\ell^{\varepsilon r})
 \frac{\rho(\ell^r)}{\ell^r}
 \ll\sum_{r>2g}\e^{r\chi} \ell^{\varepsilon r-\frac{r}{g}}
 \ll\ell^{-1-\frac{1}{2g}},
\]
where the last step uses the facts that \(\ell>\e^{8g\chi}\) and $\varepsilon<1/(2g)$. Combining these two ranges, we obtain
\[
 \sum_{r\geq1}\e^{r\chi}\min(A^r,B\ell^{\varepsilon r})
 \frac{\rho(\ell^r)}{\ell^r}
 \ll\frac1\ell.
\]
Partial summation and the prime number theorem yield
\[
 \sum_{\ell\leq x}\frac{\e^{-u}}{\ell} \ll1.
\]
Finally, we conclude that
\begin{equation}\label{eq:S3}
 S_3\ll yV_g(x)\sum_{d_1\cdots d_k\leq x}W(\mathbf d).
\end{equation}

\subsection{Estimation of $S_2$}
By \eqref{eq:Mclass}, we have
\[
 F(|Q_1(n)|,\ldots,|Q_k(n)|)\leq B|Q(n)|^\varepsilon.
\]
The bound \eqref{eq:Kl-Weil} also implies
\[
 \prod_{j=1}^k|\Kl(a;|Q_j(n)|)|^{\nu_j}
 \leq2^{\frac32|\bnu|+\sum_j\nu_j\omega(|Q_j(n)|)}
 \ll2^{|\bnu|\omega(|Q(n)|)}
 \ll|Q(n)|^\varepsilon.
\]
Here we used \(2^{|\bnu|\omega(m)}\ll m^\varepsilon\).
Together with \eqref{eq:epsilon-budget} and \eqref{eq:Q-size-uniform},
these estimates yield
\begin{equation}\label{eq:trivial-total}
 F(|Q_1(n)|,\ldots,|Q_k(n)|)
 \prod_{j=1}^k|\Kl(a;|Q_j(n)|)|^{\nu_j}
 \ll |Q(n)|^{2\varepsilon}\leq x^{\varepsilon_3}.
\end{equation}

For a prime \(q\leq z\), let \(f(q)\) be the least integer
such that \(q^{f(q)}>x^{2\varepsilon_2}\).
For \(n\in C_2\), one has \(q=P^-(b_n)\leq z\) and $d_n\leq x^{\varepsilon_2}$.
By the definition of \(d_n\), we have $d_n q^{v_q(Q(n))}>x^{\varepsilon_1}$. This further yields
\[
 q^{v_q(Q(n))}>\frac{x^{\varepsilon_1}}{d_n}
 \geq x^{\varepsilon_1-\varepsilon_2}=x^{2\varepsilon_2}.
\]
Thus \(q^{f(q)}\mid Q(n)\). Combining this with \eqref{eq:trivial-total}, we get
\begin{equation}\label{eq:S2-1}
	S_2\ll x^{\varepsilon_3}\sum_{q\leq z}
	\sum_{\substack{x<n\leq x+y\\q^{f(q)}\mid Q(n)}}1,
\end{equation}
where the sum over \(q\) is over primes. By the minimality of \(f(q)\),
\[
 x^{2\varepsilon_2}<q^{f(q)}\leq qx^{2\varepsilon_2}
 \leq x^{2\varepsilon_2+\varepsilon_3}<y.
\]
We count the solutions modulo \(q^{f(q)}\) and use
Stewart's estimate \eqref{eq:Stewart} to obtain
\[
 \sum_{\substack{x<n\leq x+y\\q^{f(q)}\mid Q(n)}}1
 \leq
 \rho(q^{f(q)})\left(\frac{y}{q^{f(q)}}+1\right)
\ll yq^{-\frac{f(q)}{g}}\leq yx^{-\frac{2\varepsilon_2}{g}}.
\]
Inserting this into \eqref{eq:S2-1}, we get
\begin{equation}\label{eq:S2}
 S_2 \ll yzx^{\varepsilon_3-\frac{2\varepsilon_2}{g}} \ll yx^{2\varepsilon_3-\frac{2\varepsilon_2}{g}}
 =yx^{-\frac{\varepsilon_2}{g}}.
\end{equation}

\subsection{Estimation of $S_4$}
We use the trivial bound \eqref{eq:trivial-total}, and obtain
\[
 S_4
 \ll x^{\varepsilon_3}
 \sum_{\substack{x^{\varepsilon_2}<d\leq x^{\varepsilon_1}\\
                  P^+(d)\leq\eta_0}}
 \sum_{\substack{x<n\leq x+y\\d\mid Q(n)}}1 \ll yx^{\varepsilon_3}
 \sum_{\substack{x^{\varepsilon_2}<d\leq x^{\varepsilon_1}\\
                  P^+(d)\leq\eta_0}}\frac{\rho(d)}d.
\]
By the multiplicativity of $\rho(d)$ and \eqref{eq:Stewart}, we have
\[
 \frac{\rho(d)}d\leq g^{\omega(d)}d^{-\frac{1}{g}}
 \ll d^{-\frac{1}{g}},
\]
since \(\omega(d)\leq \eta_0\). Rankin's trick gives
\[
\begin{aligned}
 \sum_{\substack{x^{\varepsilon_2}<d\leq x^{\varepsilon_1}\\
                  P^+(d)\leq\eta_0}}d^{-1/g}
 &\leq x^{-3\varepsilon_2/(4g)}
       \sum_{P^+(d)\leq\eta_0}d^{-1/(4g)}\\
 &=x^{-3\varepsilon_2/(4g)}
   \prod_{p\leq\eta_0}\left(1-p^{-1/(4g)}\right)^{-1}
 \ll x^{-3\varepsilon_2/(4g)}.
\end{aligned}
\]
Combining these estimates, we obtain
\begin{equation}\label{eq:S4}
 S_4\ll yx^{-\varepsilon_2/(4g)}.
\end{equation}

\subsection{Estimation of $S_5$}
By the trivial bound \eqref{eq:trivial-total}, we have
\begin{equation}\label{eq:S5}
 S_5
 \ll x^{\varepsilon_3}\sum_{1\leq d\leq x^{\varepsilon_1}}
       \sum_{\substack{x<n\leq x+y\\|Q(n)|=d}}1\ll x^{\varepsilon_1+\varepsilon_3}
 \leq yx^{-\alpha/2}.
\end{equation}

\vskip 2mm
Finally, the estimate \eqref{eq:sieve-dimension-axiom} yields
\(V_g(x)\gg(\log x)^{-g}\).
We deduce from \eqref{eq:S2}, \eqref{eq:S4} and \eqref{eq:S5} that
\[
 S_2+S_4+S_5\ll yV_g(x).
\]
Theorem \ref{thm:main} follows from this bound, \eqref{eq:S1}, \eqref{eq:S3} and the fact \(W(\mathbf1_k)=1\).

\vskip 5mm

\section{Proofs of Corollaries \ref{cor:discriminant}--\ref{cor:normalized}}\label{sec:upper-corollaries}

We follow the argument of \cite[Section~6]{Henriot-2012} for the
arithmetic sum in \eqref{eq:main-theorem}.

\begin{proof}[Proof of Corollary \ref{cor:discriminant}]
For \(1\leq j\leq k\), we write \(d_j=b_jc_j\) with
\(b_j\mid D^\infty\) and \((c_j,D)=1\).
Since \((b_1\cdots b_k,c_1\cdots c_k)=1\),
submultiplicativity \eqref{submultiplicativity} and the multiplicativity of \(H_{a,\bnu}\) yield
\begin{equation*}
\begin{aligned}
 \sum_{d_1\cdots d_k\leq x}F(\mathbf d)H_{a,\bnu}(\mathbf d)&\leq
 \bigg(\sum_{b_1\cdots b_k\mid D^\infty}
 T_{a,\bnu,F}(\mathbf b)\bigg)
 \sum_{\substack{c_1\cdots c_k\leq x\\
                  (c_1\cdots c_k,D)=1}}
 F(\mathbf c)H_{a,\bnu}(\mathbf c)\\
 &\leq
 \Delta_D(a,\bnu,F)
 \sum_{\substack{c_1\cdots c_k\leq x\\
                  (c_1\cdots c_k,D)=1}}
 F(\mathbf c)H_{a,\bnu}(\mathbf c).
\end{aligned}
\end{equation*}
Here we further apply the submultiplicativity to the prime power factors
of \(\mathbf b\) and obtain
\[
 \sum_{b_1\cdots b_k\mid D^\infty}T_{a,\bnu,F}(\mathbf b)
 \leq\prod_{p\mid D}
 \bigg(1+\sum_{|\boldsymbol\ell|\geq1}
 T_{a,\bnu,F}(p^{\boldsymbol\ell})\bigg)
 =\Delta_D(a,\bnu,F).
\]
As \(\Delta_D(a,\bnu,F)= \prod_{p|D} \Delta_p(a,\bnu,F)\), the proof is complete with Lemma~\ref{lem:T-Delta} and the last inequality of \eqref{eq:Delta-D-finite} makes use of
\(1+Cp^{-1}\leq(1-p^{-1})^{-C}\).
\end{proof}

\begin{proof}[Proof of Corollary \ref{cor:Henriot-cor2}]
For pairwise coprime \(d_1,\ldots,d_k\), each prime divides at most
one \(d_j\).
By \eqref{eq:H-majorant-coprime-D}, the inequality \(F\leq G_F\)
and the submultiplicativity of \(G_F\), we get
\begin{equation}\label{eq:c33}
 \sum_{\substack{d_1\cdots d_k\leq x\\
                  (d_1\cdots d_k,D)=1\\
                  (d_i,d_j)=1\ (i\ne j)}}
 F(\mathbf d)H_{a,\bnu}(\mathbf d)
 \leq
 \prod_{\substack{p\leq x\\p\nmid D}}
 \bigg(1+\sum_{j=1}^k\sum_{\ell\geq1}
 G_F(\mathbf p_j^\ell)H_{a,\bnu}(\mathbf p_j^\ell)
\bigg).    
\end{equation}
For \(p\big|2\lc(Q)\) or $p\le g$,  we have 
\begin{align*}
1+ \sum_{j=1}^k\sum_{\ell\geq1}
 G_F(\mathbf p_j^\ell)H_{a,\bnu}(\mathbf p_j^\ell)
 = \Delta_p(a,\bnu,F)\le 1+Cp^{-1},
\end{align*}
by Lemma~\ref{lem:T-Delta} with $\boldsymbol{\ell}= (0,\ldots,\ell,\ldots,0)$.  As \(S= \prod_{p\in \mathcal{S}} p\), we get, by \eqref{eq:Delta}, 
\[
\Delta_S(a,\bnu,F) \le \prod_{p|S}\big(1+Cp^{-1}\big) \le \prod_{p|S}\big(1-p^{-1}\big)^{-C} = \Big(\frac{S}{\varphi(S)}\Big)^C.
\]

For \(p\nmid S\), we apply Case (2) of Lemma~\ref{lem:local-majorant},  and 
\[
 \sum_{j=1}^k\sum_{\ell\geq2}
G_F(\mathbf p_j^\ell)H_{a,\bnu}(\mathbf p_j^\ell)
 \le
  B2^{\nu_j}\sum_{j=1}^k \rho_{Q_j}(p)
 \sum_{\ell\geq2}p^{-(1-\varepsilon)\ell}
 \ll p^{-3/2}.
\]
In light of the positivity of summands, 
\[
1+\sum_{j=1}^k\sum_{\ell\geq1}
 G_F(\mathbf p_j^\ell)H_{a,\bnu}(\mathbf p_j^\ell)
 \le \big(1+O(p^{-3/2})\big)\bigg(1+\sum_{j=1}^k
 G_F(\mathbf p_j^1)H_{a,\bnu}(\mathbf p_j^1)
\bigg)  
\]
As \(\prod_p(1+O(p^{-3/2}))\ll 1\) and 
\[H_{a,\bnu}(\mathbf p_j^1) =  \frac{\rho_{Q_j}(p)}p \lambda_{a,\nu_j}(p),\]
the result follows from
Corollary~\ref{cor:discriminant} and \eqref{eq:c33}.
\end{proof}

\begin{proof}[Proof of Theorem~\ref{thm:prime-variable}]
We use the argument of Nair and Tenenbaum
\cite[Section~9]{NT-1998}. Define
\(\chi_x(n)=\mathbf1_{P^-(n)>x}\) and
\(F_0(\mathbf d,d_{k+1})=F(\mathbf d)\chi_x(d_{k+1})\). One has \(F_0\in\mathcal M_{k+1}(A,B,\varepsilon)\). We set
\[
 \widehat{\mathbf Q}=(Q_1,\ldots,Q_k,X),\qquad
 \widehat Q=XQ,\qquad
 \bnu_0=(\bnu,0).
\]
Since \(Q(0)\ne0\), the polynomials in \(\widehat{\mathbf Q}\) are pairwise coprime and irreducible.
Moreover, \(\widehat Q\) is primitive with
\(\deg\widehat Q=g+1\) and \(\|\widehat Q\|=\|Q\|\). Thus
\eqref{eq:prime-epsilon-range} permits an application of
Theorem~\ref{thm:main} to the enlarged tuple.

For \(x<n\leq x+y\leq2x\), we have \(\chi_x(n)=1\) precisely when \(n\) is prime. Applying
Corollary~\ref{cor:Henriot-cor2}, the sum on the left-hand side of \eqref{eq:prime-variable} is at most
\begin{equation}\label{eq:prime-density-bound}
 \Delta_S(a,\bnu,F_0)\,y
 \widehat V_{g+1}(x)
  \prod_{\substack{p\le x\\p\nmid S}}
 \left(
 1+
 \sum_{j=1}^{k+1}
 \widehat{G}_{F_0}(\mathbf p_j^1)
 \frac{\rho_{Q_j}(p)
       \lambda_{a,\nu_j}(p)}{p}
 \right)
\end{equation}
where \(\widehat{V}_{g+1}(x)\) is defined as in \eqref{eq:Vg} with $\rho_{\widehat{Q}}$ in place of  $\rho$ and  \(\widehat{G}_{F_0}\) is defined by \eqref{eq:standard-majorant} for \(M_{k+1}(A,B,\varepsilon)\) and $Q_{k+1}(X)= X$. 
Recalling $\rho= \rho_Q$, we have 
\(\rho_{\widehat Q}(p)=\rho(p)+\mathbf1_{p\nmid Q(0)}\); thus for
\(p>g+1\) with \(p\nmid Q(0)\),
\[
 1-\frac{\rho_{\widehat{Q}}(p)}p
 \leq\left(1-\frac{\rho(p)}p\right)
       \left(1-\frac1p\right).
\]
The factors with \(p\mid Q(0)\) are unchanged. Mertens' theorem
therefore gives
\[
 \widehat V_{g+1}(x)
 \ll_g \frac{|Q(0)|}{\varphi(|Q(0)|)}
       \frac{V_g(x)}{\log x}.
\]
Combining this estimate with \eqref{eq:prime-density-bound} and  \eqref{eq:Vg} proves \eqref{eq:prime-variable}.
\end{proof}

\begin{proof}[Proof of Corollary \ref{cor:normalized}]
We take \(\varepsilon=\alpha/(100g(g+1/\delta))\) and
\(F(\mathbf d)=\prod_{j=1}^k z_j^{\omega(d_j)}\).
We have \(F\in\mathcal M_k(A,B,\varepsilon)\), where
\(A=1+\max(z_1,\ldots,z_k)\) and \(B\) depends only on
\(\mathbf z\) and \(\varepsilon\).
Moreover, \(G_F=F\) and \(G_F(\mathbf p_j^1)=z_j\).

We apply Corollary~\ref{cor:Henriot-cor2} and absorb
\(\Delta_S(a,\bnu,F)\)  into the implied constant. This constant is
independent of \(a\), by
\eqref{eq:vartheta-bad}.

For \(p\notin\mathcal S\), the \(Q_j\)'s have no common root modulo
\(p\), so \(\rho(p)=\sum_j\rho_{Q_j}(p)\).
Invoking \eqref{eq:lambda-prime-upper}, we obtain
\[
 \bigg(1-\frac{\rho(p)}p\bigg)
 \bigg(1+\sum_{j=1}^k z_j
       \frac{\rho_{Q_j}(p)\lambda_{a,\nu_j}(p)}p\bigg)\leq
 \exp\bigg(\frac1p\sum_{j=1}^k
       (z_j\sn(\nu_j)-1)\rho_{Q_j}(p)
       +O_{\bnu,\mathbf z,g}(p^{-9/8})\bigg).
\]
The error term is uniform in \(a\).
We now use Erd\H{o}s's estimate \cite[Lemma~7]{Erdos-1952},
\[
 \sum_{p\leq x}\frac{\rho_{Q_j}(p)}p
 =\log\log x+O_{Q_j}(1),
\]
and then obtain the required
estimate
\[
 \prod_{\substack{p\leq x\\p\notin\mathcal S}}
 \left(1-\frac{\rho(p)}p\right)
 \left(1+\sum_{j=1}^k z_j
       \frac{\rho_{Q_j}(p)\lambda_{a,\nu_j}(p)}p\right)
 \ll (\log x)^{\sum_{j=1}^k(z_j\sn(\nu_j)-1)}.
 \qedhere
\]
\end{proof}

\section{Proof of Theorem \ref{thm:lower}}\label{sec:lower-proof}

Fix \(z>0\), \(\nu\geq1\), and \(A>0\), and suppose that
\(0<|a|\leq x^A\). Throughout this section, \(x\) is sufficiently
large in terms of \(A,z,\nu\). Set
\[
 C=4(1+z\sn(\nu)),\qquad Z=x^{\frac{1}{30C}}.
\]
Let \(\mathcal D\) consist of the squarefree integers \(d\) satisfying
\begin{equation*}
 d\leq x^{\frac{1}{30}},\qquad \omega(d)\leq C\log\log x,\qquad
 P^+(d)\leq Z,\qquad (d,2a)=1.
\end{equation*}
For convenience, we extend \(\lambda_{a,\nu}\) to a multiplicative function over squarefree integers:
\[
 \lambda_{a,\nu}(d)=\prod_{p\mid d}\lambda_{a,\nu}(p),
 \qquad \lambda_{a,\nu}(1)=1.
\]

\subsection{The contribution of the small primes}

\begin{lemma}\label{lem:lower-small-mass}
We have
\[
 \mathcal H_a:=\sum_{d\in\mathcal D}
 \frac{z^{\omega(d)}\lambda_{a,\nu}(d)}d
 \gg
 (\log x)^{z\sn(\nu)}
 \prod_{p\mid2a}\left(1+\frac{z\sn(\nu)}p\right)^{-1},
\]
where the implied constant depends only on \(z,\nu\).
\end{lemma}

\begin{proof}
Let \(\mathcal B\) be the set of squarefree integers \(d\) with
\(P^+(d)\leq Z\) and \((d,2a)=1\), and denote
\[
 \mathcal E_a=\sum_{d\in\mathcal B}
 \frac{z^{\omega(d)}\lambda_{a,\nu}(d)}d
 =\prod_{\substack{p\leq Z\\p\nmid2a}}
 \left(1+\frac{z\lambda_{a,\nu}(p)}p\right).
\]
For each prime \(p\leq Z\) with \(p\nmid2a\), by multiplicativity, we have
\[
\begin{aligned}
 \sum_{\substack{d\in\mathcal B\\p\mid d}}
 \frac{z^{\omega(d)}\lambda_{a,\nu}(d)}d
 &=\frac{z\lambda_{a,\nu}(p)}p
 \prod_{\substack{q\leq Z\\q\nmid2a,\ q\ne p}}
 \left(1+\frac{z\lambda_{a,\nu}(q)}q\right)\\
 &\leq\frac{z\lambda_{a,\nu}(p)}p\,\mathcal E_a.
\end{aligned}
\]
Note that \(\log d=\sum_{p\mid d}\log p\) and
\(\omega(d)=\sum_{p\mid d}1\) for \(d\in\mathcal B\).
For \(d\in\mathcal B\setminus\mathcal D\), one has
\(d>x^{1/30}\) or \(\omega(d)>C\log\log x\). Hence
\begin{equation}\label{eq:EH}
	\begin{aligned}
		\frac{\mathcal E_a-\mathcal H_a}{\mathcal E_a}
		&\leq\frac1{\mathcal E_a}
		\sum_{d\in\mathcal B}
		\frac{z^{\omega(d)}\lambda_{a,\nu}(d)}d
		\left(\frac{30\log d}{\log x}
		+\frac{\omega(d)}{C\log\log x}\right)\\
		&\leq
		\sum_{\substack{p\leq Z\\p\nmid2a}}
		\frac{z\lambda_{a,\nu}(p)}p
		\left(\frac{30\log p}{\log x}
		+\frac1{C\log\log x}\right).
	\end{aligned}
\end{equation}
Using \eqref{eq:lambda-prime}, the relation
\(\log Z=(\log x)/(30C)\), and Mertens' estimates
\[
 \sum_{p\leq Z}\frac{\log p}p=\log Z+O(1),
 \qquad
 \sum_{p\leq Z}\frac1p=\log\log Z+O(1),
\]
we bound the last sum in \eqref{eq:EH} by
\[
 \frac{30z\sn(\nu)\log Z+O_{z,\nu}(1)}{\log x}
 +\frac{z\sn(\nu)\log\log Z+O_{z,\nu}(1)}{C\log\log x}
 =\frac{2z\sn(\nu)}C+o_{z,\nu}(1).
\]
Since \(2z\sn(\nu)/C<1/2\), we have
\(\mathcal H_a\gg_{z,\nu}\mathcal E_a\). Moreover, by \eqref{eq:lambda-prime} and Mertens' formula, we get
\begin{equation*}
	\mathcal E_a\gg\prod_{\substack{p\leq Z\\p\nmid2a}}
	\left(1+\frac{z\sn(\nu)}p\right)\gg_{z,\nu} (\log Z)^{z\sn(\nu)}
	\prod_{p\mid2a}\left(1+\frac{z\sn(\nu)}p\right)^{-1}.
\end{equation*}
This proves the lemma.
\end{proof}

\subsection{Averaging over one large prime}

For \(d\in\mathcal D\) and \((r,d)=1\), define
\[
 W_d(r)=z^{\omega(d)}|\Kl(a\bar r^{\,2};d)|^\nu,
 \qquad W_1(r)=1.
\]
This weight depends only on \(r\bmod d\). By twisted multiplicativity,
we have
\[
 W_d(r)=\prod_{p\mid d}
 z\left|\Kl\left(a\overline{r d/p}^{2};p\right)\right|^\nu.
\]
Averaging over the reduced residue classes modulo \(d\) and applying
the Chinese remainder theorem, we obtain from \eqref{eq:lambda}
\begin{equation}\label{eq:lower-small-mean}
 \frac{1}{\varphi(d)}\sideset{}{^*}\sum_{r\bmod d}W_d(r)
 =z^{\omega(d)}\lambda_{a,\nu}(d).
\end{equation}

Consider primes \(p_2,p_3\) in the ranges
\begin{equation}\label{eq:lower-large-ranges}
 x^{\frac{3}{10}}\leq p_2<\frac{1}{4}x^{\frac{1}{3}},\qquad
 x^{\frac{1}{4}}\leq p_3<\frac{1}{4}x^{\frac{3}{10}},
\end{equation}
and primes \(p_1\) such that
\begin{equation*}\label{eq:lower-product-band}
 \frac{x}{2}<dp_1p_2p_3\leq x.
\end{equation*}
For \(d\in\mathcal D\), these conditions imply
\begin{equation}\label{eq:lower-p1-range}
 8x^{\frac{1}{3}}<p_1\leq x^{\frac{9}{20}}.
\end{equation}
Thus \(p_1>p_2>p_3>Z\), and \((p_1p_2p_3,d)=1\).
Moreover, as \(P^+(d)\le Z\) , distinct tuples \((d,p_1,p_2,p_3)\) give
distinct integers $n=dp_1p_2p_3$.
Throughout this subsection, sums over \(p_2\) and \(p_3\) are restricted
to \eqref{eq:lower-large-ranges}. Define
\begin{equation}\label{eq:lower-total-weight}
 \mathcal T_a=
 \sum_{d\in\mathcal D}\sum_{p_2}\sum_{p_3}
 \sum_{\frac{x}{2dp_2p_3}<p_1\leq\frac{x}{dp_2p_3}}
 W_d(p_1p_2p_3).
\end{equation}
Set \(B_0=C\max\{\log(z2^\nu),0\}\).
Since every prime dividing \(d\) is odd, \eqref{eq:prime-power-Kl-bound}
and the restriction on \(\omega(d)\) give
\begin{equation}\label{eq:lower-weight-bound}
 0\leq W_d(r)\leq(z2^\nu)^{\omega(d)}
 \leq(\log x)^{B_0}
 \qquad(d\in\mathcal D,\ (r,d)=1).
\end{equation}

\begin{lemma}\label{lem:lower-prime-average}
We have
\[
 \mathcal T_a\gg
 x(\log x)^{z\sn(\nu)-1}
 \prod_{p\mid2a}\left(1+\frac{z\sn(\nu)}{p}\right)^{-1},
\]
where the implied constant depends only on \(z,\nu\).
\end{lemma}

\begin{proof}
Let \(\Theta_a\) be the sum in \eqref{eq:lower-total-weight} with the
additional factor \(\log p_1\). Since
\(\Theta_a\leq(\log x)\mathcal T_a\), it suffices, by
Lemma~\ref{lem:lower-small-mass}, to prove
\(\Theta_a\gg_{z,\nu}x\mathcal H_a\).

Fix \(p_2,p_3\) and \(d\in\mathcal D\). The primes \(p_1\) in
\eqref{eq:lower-total-weight} are coprime to \(d\). Splitting the sum
over \(p_1\) into reduced residue classes modulo \(d\), we rewrite the inner sum of \(\Theta_a\) into
\[
\begin{aligned}
 \sum_{\frac{x}{2}<dp_1p_2p_3\leq x}(\log p_1)W_d(p_1p_2p_3)
 &=\sideset{}{^*}\sum_{r\bmod d}W_d(rp_2p_3)
   \sum_{\substack{x/2<dp_1p_2p_3\leq x\\
                   p_1\equiv r\,({\rm mod}\,d)}}\log p_1\\
 &=\frac{x}{2dp_2p_3}\,z^{\omega(d)}\lambda_{a,\nu}(d)
   +E(d;p_2,p_3),
\end{aligned}
\]
by \eqref{eq:lower-small-mean} and \((p_2p_3,d)=1\), where 
\[
 E(d;p_2,p_3)=
 \sideset{}{^*}\sum_{r\bmod d}W_d(rp_2p_3)
 \left(
 \sum_{\substack{x/2<dp_1p_2p_3\leq x\\
                 p_1\equiv r \,({\rm mod}\, d)}}\log p_1
 -\frac{x}{2dp_2p_3\varphi(d)}
 \right).
\]
Using \eqref{eq:lower-weight-bound}, we get
\[
 E(d;p_2,p_3)\ll (\log x)^{B_0}\varphi(d) \max_{(r,d)=1} \max_{y\leq x/(dp_2p_3)}\bigg| \sum_{\substack{p\leq y\\
 		p \equiv r \,({\rm mod}\, d)}}\log p
 -\frac{y}{\varphi(d)}\bigg|.
\]

To sum these errors, we partition \(d\leq x^{1/30}\) into dyadic blocks
\(D\leq d<2D\). By \eqref{eq:lower-large-ranges} and \(1\leq D\leq x^{1/30}\),
we have \(16x^{1/3}<x/(Dp_2p_3) \leq x^{9/20}\).
The Bombieri--Vinogradov theorem \cite[Chapter~17]{IK-2004} therefore gives
\[
\begin{aligned}
	 \sum_{d\in\mathcal D}|E(d;p_2,p_3)|& \ll (\log x)^{B_0+1} \max_{D\leq x^{\frac{1}{30}}} D  \sum_{\substack{d\in\mathcal D\\D\leq d<2D}} \max_{(r,d)=1} \max_{y\leq x/(Dp_2p_3)}\bigg| \sum_{\substack{p\leq y\\
	 		p \equiv r \,({\rm mod}\, d)}}\log p
	 -\frac{y}{\varphi(d)}\bigg|\\
& \ll \frac{x}{p_2p_3(\log x)^{2}}.
\end{aligned}
\]

Recall that  \(\Theta_a\) is defined as in \eqref{eq:lower-total-weight} with $W_d(p_1p_2p_3)\log p_1$ in place of $W_d(p_1p_2p_3)$. It follows from Mertens' estimate that the contribution from $E(d;p_2,p_3)$ to  $ \Theta_a$ is $O(x/(\log x)^{2})$.
We therefore conclude that
\[
 \Theta_a=
 \frac{x}{2}\left(\sum_{p_2}\frac{1}{p_2}\right)
          \left(\sum_{p_3}\frac{1}{p_3}\right)\mathcal H_a
 +O_{z,\nu}\left(\frac{x}{(\log x)^2}\right).
\]
Since \(\mathcal H_a\geq1\), the $O_{z,\nu}$-term  is
\(o(x\mathcal H_a)\). Thus \(\Theta_a\gg_{z,\nu}x\mathcal H_a\),
as required.
\end{proof}

We shall also need the bound
\begin{equation}\label{eq:lower-bad-prime-product}
 \prod_{p\mid2a}\left(1+\frac{z\sn(\nu)}{p}\right)^{-1}
 \gg_{A,z,\nu}(\log\log(3x))^{-z\sn(\nu)}.
\end{equation}
Indeed, for \(m\geq2\),
\[
 \prod_{p\mid m}\left(1+\frac{z\sn(\nu)}{p}\right)
 \leq\left(\frac{m}{\varphi(m)}\right)^{z\sn(\nu)}
 \ll_{z,\nu}(\log\log(3m))^{z\sn(\nu)}
\]
by using the standard bound for \(m/\varphi(m)\).
Taking \(m=2|a|\leq2x^A\) gives
\eqref{eq:lower-bad-prime-product}.

\subsection{The angles at the large prime factors}

For an odd prime \(p\) and \(t\in\F_p^\times\), write
\(\Kl(t;p)=2\cos\theta_{p,t}\), with \(\theta_{p,t}\in[0,\pi]\),
as in the proof of Lemma~\ref{lem:lambda}.
For an interval \(I\subset[0,\pi]\), let
\[
 \mu_{\rm ST}(I)=\frac2\pi\int_I\sin^2\theta\,d\theta.
\]
The following lemma adapts \cite[Proposition~4.1]{FM-2003} to
the parameter \(a\) and the interval \(Y/2<pm\leq Y\).

\begin{lemma}\label{lem:lower-large-sieve}
Let \(P>1\), \(N\geq1\), \(Y\geq4\), and let \(a\ne0\) be an
integer. Let \(\mathcal P\subset[P,2P)\) be a set of primes
not dividing \(2a\), and let \(f:\N\to\R_{\geq0}\) be supported
on \([1,N]\). Suppose that \((m,p)=1\) whenever \(f(m)\ne0\)
and \(p\in\mathcal P\). Then, for any interval \(I\subset[0,\pi]\),
\[
\begin{aligned}
 \sum_{p\in\mathcal P}
  \sum_{\substack{Y/2<pm\leq Y\\
                   \theta_{p,a\bar m^2}\in I}}f(m)
= &\left(\mu_{\rm ST}(I)+O(P^{-\frac{1}{8}})\right)
       \sum_{p\in\mathcal P}\sum_{Y/2<pm\leq Y}f(m)\\
 &+O\left((\log Y)(P\#\mathcal P)^{\frac{1}{2}}
       \left(\frac NP+P\right)^{\frac{1}{2}}\Big(\sum_m|f(m)|^2\Big) ^{\frac{1}{2}}\right),
\end{aligned}
\]
The implied constants are absolute.
\end{lemma}

\begin{proof}
For each \(p\in\mathcal P\), we split the sum over \(m\) into
nonzero residue classes modulo \(p\) and get
\begin{equation}\label{eq:lower-residue-decomposition}
\begin{aligned}
 \sum_{\substack{Y/2<pm\leq Y\\
                  \theta_{p,a\bar m^2}\in I}}f(m)
 &=\sum_{\substack{1\leq r<p\\\theta_{p,a\bar r^2}\in I}}
       \sum_{\substack{Y/2<pm\leq Y\\
                       m\equiv r\,({\rm mod}\,p)}}f(m)\\
 &=\frac1{p-1}
       \sum_{\substack{1\leq r<p\\\theta_{p,a\bar r^2}\in I}}
       \sum_{Y/2<pm\leq Y}f(m)\\
 &\quad+
       \sum_{\substack{1\leq r<p\\\theta_{p,a\bar r^2}\in I}}
       \bigg(
       \sum_{\substack{Y/2<pm\leq Y\\
                       m\equiv r\,({\rm mod}\,p)}}f(m)
       -\frac1{p-1}\sum_{Y/2<pm\leq Y}f(m)
       \bigg)\\
&= \Sigma_{1,p}+ \Sigma_{2,p}, \mbox{ say}.
\end{aligned}
\end{equation}
By \eqref{eq:a-ST-discrepancy} and the nonnegativity of \(f\),
the sum of $\Sigma_{1,p}$ over \(p\in\mathcal P\) equals
\[
\sum_{p\in\mathcal P}
  \left(\mu_{\rm ST}(I)+O(p^{-\frac{1}{8}})\right)
  \sum_{Y/2<pm\leq Y}f(m)=\left(\mu_{\rm ST}(I)+O(P^{-\frac{1}{8}})\right)
  \sum_{p\in\mathcal P}\sum_{Y/2<pm\leq Y}f(m),
\]
uniformly in \(a\) and \(I\).

The sum  of $\Sigma_{2,p}$ over \(p\in\mathcal P\) is, by the Cauchy--Schwarz inequality, 
\[
\begin{aligned}
 &\ll(P\#\mathcal P)^{\frac{1}{2}}
\bigg(\sum_{p\in\mathcal P}\sideset{}{^*}\sum_{r\bmod p}
\bigg|
\sum_{\substack{Y/2<pm\leq Y\\
		m\equiv r\,({\rm mod}\,p)}}f(m)
-\frac1{p-1}\sum_{Y/2<pm\leq Y}f(m)
\bigg|^2\bigg)^{\frac{1}{2}}\\
& \ll(\log Y)(P\#\mathcal P)^{\frac{1}{2}}
\left(\frac NP+P\right)^{\frac{1}{2}}
\left(\sum_m|f(m)|^2\right)^{\frac{1}{2}},
\end{aligned}
\]
where we used \cite[Lemma~2.5]{FM-2003} in the last step.
Combining the two estimates proves the lemma.
\end{proof}

Let \(\mathcal T_a^\ast\) denote the sum in
\eqref{eq:lower-total-weight} restricted by
\((p_1p_2p_3,2a)=1\). Since \(2|a|\leq2x^A\), only
\(O_A(1)\) primes \(p\geq x^{1/4}\) divide \(2a\). Thus
\eqref{eq:lower-weight-bound} gives
\begin{equation}\label{eq:lower-bad-large-primes}
0\leq\mathcal T_a-\mathcal T_a^\ast
\ll (\log x)^{B_0}
 \sum_{\substack{p\mid2a\\p\geq x^{\frac{1}{4}}}}\frac{x}{p} 
 \ll_A x^{\frac{3}{4}}(\log x)^{B_0}.
\end{equation}
By \eqref{eq:lower-bad-large-primes}, Lemma~\ref{lem:lower-prime-average}, and
\eqref{eq:lower-bad-prime-product}, we have, uniformly in \(a\), 
\begin{equation}\label{eq:T}
\mathcal T_a^\ast=(1+o(1))\mathcal T_a \ge x^{1-o(1)}
\end{equation}

For the tuples counted by \(\mathcal T_a^\ast\), define
\(\theta_j\in[0,\pi]\) by
\[
 \Kl\left(a\overline{d\prod_{i\ne j}p_i}^{\,2};p_j\right)
 =2\cos\theta_j\qquad(1\leq j\leq3).
\]
For an interval \(I\subset[0,\pi]\), define
\[
 \mathcal T_{a,j}(I)=
 \sum_{d\in\mathcal D}\sum_{p_2}\sum_{p_3}
 \sum_{\substack{
       \frac{x}{2dp_2p_3}<p_1\leq\frac{x}{dp_2p_3}\\
       (p_1p_2p_3,2a)=1\\
       \theta_j\in I}}
 W_d(p_1p_2p_3),
\]
where \(p_2,p_3\) satisfy \eqref{eq:lower-large-ranges}.

\begin{lemma}\label{lem:lower-large-prime-ST}
As \(x\to\infty\), we have
\[
 \mathcal T_{a,j}(I)
 =\mu_{\rm ST}(I)\mathcal T_a^\ast+o(\mathcal T_a^\ast),
\]
uniformly for \(0<|a|\leq x^A\), \(1\leq j\leq3\), and
intervals \(I\subset[0,\pi]\).
\end{lemma}

\begin{proof}
Fix \(j\in\{1,2,3\}\). Throughout the proof, \(p_2,p_3\)
satisfy \eqref{eq:lower-large-ranges}, \(p_1\) satisfies
\eqref{eq:lower-p1-range}, and \(p_i\nmid2a\) for
\(1\leq i\leq3\). These three intervals are disjoint and lie
above \(Z\).

For each \(d\in\mathcal D\), we split the sum over \(p_j\) in
\(\mathcal T_{a,j}(I)\) into \(O(\log x)\) dyadic intervals \([P,2P)\) and reduced residue
classes modulo \(d\). Fix \(P\) and a reduced residue class
\(b\bmod d\), and let \(\mathcal P\) denote the corresponding
set of primes. Here \(x^{1/4}\ll P\ll x^{9/20}\) and
\(d\leq x^{1/30}\). For the application of Lemma~\ref{lem:lower-large-sieve}, define
\[
f(m)=
\begin{cases}
	W_d\!\left(\frac{bm}{d}\right),
	&\displaystyle m=d\prod_{i\ne j}p_i\leq\frac{x}{P},\\[2mm]
	0,&\text{otherwise}
\end{cases}.
\]
We use the unique representation $m=dp_1p_2p_3$. For
\(m\) in the support of \(f\) and \(p\in\mathcal P\), we have
\((m/d,d)=(m,p)=1\). Since \(p\equiv b\pmod d\), the
periodicity of \(W_d\) yields
\[
f(m)=W_d\!\left(\frac{bm}{d}\right)
=W_d\!\left(\frac{pm}{d}\right)
=W_d(p_1p_2p_3),
\qquad p_j=p.
\]
Thus, we have
\begin{equation}\label{eq:lower-angle-decomposition}
\begin{aligned}
 \mathcal T_{a,j}(I)
 &=\sum_{d\in\mathcal D}\sum_P
   \sideset{}{^*}\sum_{b\bmod d}
   \sum_{p\in\mathcal P}
   \sum_{\substack{\frac{x}{2}<pm\leq x\\
                   \theta_{p,a\bar m^2}\in I}}f(m),\\
 \mathcal T_a^\ast
 &=\sum_{d\in\mathcal D}\sum_P
   \sideset{}{^*}\sum_{b\bmod d}
   \sum_{p\in\mathcal P}
   \sum_{\frac{x}{2}<pm\leq x}f(m).
\end{aligned}
\end{equation}
Here \(\mathcal P\) and \(f\) depend on \(d,P,b\).

We obtain from \eqref{eq:lower-weight-bound} that
\begin{equation}\label{eq:second-f}
	 \sum_m|f(m)|^2
	\leq(\log x)^{2B_0}
	\sum_{\substack{m\leq\frac{x}{P}\\d\mid m}}1
	\leq\frac{x}{Pd}(\log x)^{2B_0}.
\end{equation}
With \(N=x/P\) and \(Y=x\), Lemma~\ref{lem:lower-large-sieve} yields
\begin{equation}\label{eq:lower-angle-block}
	\begin{aligned}
		\sum_{p\in\mathcal P}
		\sum_{\substack{\frac{x}{2}<pm\leq x\\
				\theta_{p,a\bar m^2}\in I}}f(m)=&\left(\mu_{\rm ST}(I)+O(P^{-\frac{1}{8}})\right)
		\sum_{p\in\mathcal P}\sum_{\frac{x}{2}<pm\leq x}f(m)
		+O\left(\frac{x^{\frac{19}{20}}(\log x)^{B_0+1}}{d}\right)\\
		=&\mu_{\rm ST}(I)\sum_{p\in\mathcal P}\sum_{\frac{x}{2}<pm\leq x}f(m)
		+O\left(\frac{x^{\frac{31}{32}}(\log x)^{B_0}}{d^2}+\frac{x^{\frac{19}{20}}(\log x)^{B_0+1}}{d}\right).
	\end{aligned}
\end{equation}
Indeed, the error in Lemma~\ref{lem:lower-large-sieve}
is bounded by
\[
(\log x)(P\#\mathcal P)^{\frac{1}{2}}
\left(\frac{x}{P^2}+P\right)^{\frac{1}{2}}
\left(\sum_m|f(m)|^2\right)^{\frac{1}{2}}
\ll\frac{x^{\frac{19}{20}}(\log x)^{B_0+1}}{d},
\]
where we used the bound \eqref{eq:second-f}, \(\#\mathcal P\ll P/d\) and
\(x^{1/4}\ll P\ll x^{9/20}\).
By using \(d\mid m\) for $m$ on the support
of \(f\) and  \eqref{eq:lower-weight-bound}, the error term coming from \(O(P^{-1/8})\) is bounded by
\[
P^{-\frac18}\sum_{p\in\mathcal P}\sum_{\frac{x}{2}<pm\leq x}f(m)
\ll P^{-\frac18}\#\mathcal P\,\frac{x}{Pd}(\log x)^{B_0}
\ll\frac{x^{\frac{31}{32}}(\log x)^{B_0}}{d^2}.
\]

We now sum \eqref{eq:lower-angle-block} over \(b\), the dyadic
intervals, and \(d\in\mathcal D\). The errors contribute
\[
 \ll\sum_{d\in\mathcal D}\sum_P
       \sideset{}{^*}\sum_{b\bmod d}\left(\frac{x^{\frac{31}{32}}(\log x)^{B_0}}{d^2}+\frac{x^{\frac{19}{20}}(\log x)^{B_0+1}}{d} \right)       
 \ll x^{\frac{59}{60}}(\log x)^{B_0+2}.
\]
By \eqref{eq:lower-angle-decomposition} and \eqref{eq:T}, we obtain
\[
 \left|\mathcal T_{a,j}(I)-\mu_{\rm ST}(I)\mathcal T_a^\ast\right|
 \ll_{A,z,\nu}x^{\frac{59}{60}}(\log x)^{B_0+2}
 =o(\mathcal T_a^\ast).
\]
\end{proof}

\subsection{Completion of the proof}
Let
\[
 I_0=\{\theta\in[0,\pi]:|\cos\theta|<1/4\}.
\]
Its Sato--Tate measure satisfies
\[
 \mu_{\rm ST}(I_0)
 =\frac2\pi\left(\arcsin\frac14+\frac{\sqrt{15}}{16}\right)
 <\frac13.
\]
We sum the estimates of Lemma~\ref{lem:lower-large-prime-ST}
over \(1\leq j\leq3\).
Among the tuples counted by \(\mathcal T_a^\ast\), those with some
\(\theta_j\in I_0\) have total weight at most
\((3\mu_{\rm ST}(I_0)+o(1))\mathcal T_a^\ast\).
Since \(3\mu_{\rm ST}(I_0)<1\), the tuples with
\((p_1p_2p_3,2a)=1\) and
\(|2\cos\theta_j|\geq1/2\) for \(1\leq j\leq3\)
have total weight \(\gg\mathcal T_a^\ast\).

For such a tuple, \(n=dp_1p_2p_3\) is squarefree,
\((n,2a)=1\), and \(x/2<n\leq x\).
Twisted multiplicativity gives
\[
 z^{\omega(n)}|\Kl(a;n)|^\nu
 =W_d(p_1p_2p_3)\prod_{j=1}^3z|2\cos\theta_j|^\nu
 \geq z^3 2^{-3\nu}W_d(p_1p_2p_3).
\]
Distinct tuples yield distinct integers \(n\). We sum this inequality
and apply \(\mathcal T_a^\ast\sim\mathcal T_a\) and
Lemma~\ref{lem:lower-prime-average}. Replacing the product over
\(p\mid2a\) by that over \(p\mid a\) costs at most the fixed factor
\(1+z\sn(\nu)/2\). Thus
\[
 \sum_{\substack{x/2<n\leq x\\n\ \mathrm{squarefree}\\(n,2a)=1}}
 z^{\omega(n)}|\Kl(a;n)|^\nu
 \gg_{A,z,\nu}
 \prod_{p\mid a}\left(1+\frac{z\sn(\nu)}p\right)^{-1}
 x(\log x)^{z\sn(\nu)-1}.
\]

\section{Proofs of the applications}\label{sec:application-proofs}

\begin{proof}[Proof of Corollary~\ref{cor:moment-order}]
For the upper bound, we apply Corollary~\ref{cor:normalized} with
\(k=1\), \(Q_1(X)=X\), \(z_1=z\) and \(\nu_1=\nu\).
For sufficiently large \(x\), we have
\[
 \sum_{x<n\leq2x}z^{\omega(n)}|\Kl(a;n)|^\nu
 \ll x(\log x)^{z\sn(\nu)-1},
\]
where the implied constant depends only on \(z,\nu\). The upper bound then follows by a dyadic
decomposition. The lower bound follows directly from Theorem~\ref{thm:lower}.
\end{proof}

\begin{proof}[Proof of Corollary~\ref{cor:Kl-large-values}]
For every \(n\in\mathcal L_{a,B}(x)\), one has
\(|\Kl(a;n)|^{\nu_B}\geq (\log x)^{B\nu_B}\). By Corollary~\ref{cor:moment-order}, we have
\begin{equation}\label{eq:upper-LaB}
	 \#\mathcal L_{a,B}(x)
	\leq (\log x)^{-B\nu_B}
	\sum_{x/2<n\leq x}|\Kl(a;n)|^{\nu_B}
	\ll x(\log x)^{-I(B)}.
\end{equation}

Fix \(\varepsilon>0\). For \(\tau>0\), write
\(\nu:=\nu_B+\tau\) and \(\beta:=\sn'(\nu+\tau)\). As
\(\tau\to0^+\), we have
\(1+\nu\beta-\sn(\nu)\longrightarrow I(B)\).
We may therefore choose \(\tau=\tau(B,\varepsilon)>0\) so small that
\[
 1+\nu\beta-\sn(\nu)<I(B)+\varepsilon.
\]
Since \(\sn'(\nu_B)=B\), the strict convexity of \(\sn\) gives
\begin{equation}\label{eq:B-beta}
 \beta>B,
 \qquad
 \frac{\sn(\nu)-\sn(\nu-\tau)}{\tau}>B,
 \qquad
 \frac{\sn(\nu+\tau)-\sn(\nu)}{\tau}<\beta.
\end{equation}
By Corollary~\ref{cor:moment-order} and \eqref{eq:B-beta}, there is a
constant \(c>0\) such that
\[
\begin{aligned}
 \sum_{\substack{x/2<n\leq x\\
 |\Kl(a;n)|\notin[(\log x)^B,(\log x)^\beta]}}
 |\Kl(a;n)|^\nu
 &\leq
 (\log x)^{B\tau}\sum_{n\leq x}|\Kl(a;n)|^{\nu-\tau}
 +(\log x)^{-\beta\tau}\sum_{n\leq x}|\Kl(a;n)|^{\nu+\tau}\\
 &\ll
 x(\log x)^{\sn(\nu-\tau)-1+B\tau}
 +x(\log x)^{\sn(\nu+\tau)-1-\beta\tau}\\
 &\ll x(\log x)^{\sn(\nu)-1-c}.
\end{aligned}
\]
On the other hand, Theorem~\ref{thm:lower}, with \(z=1\) and exponent
\(\nu\), gives
\[
 \sum_{\substack{x/2<n\leq x\\n\ \mathrm{squarefree}\\(n,2a)=1}}
 |\Kl(a;n)|^\nu
 \gg x(\log x)^{\sn(\nu)-1}.
\]
It follows that
\begin{equation}\label{eq:large-value-mass}
 x(\log x)^{\sn(\nu)-1}
 \ll\sum_{\substack{x/2<n\leq x\\
                    n\ \mathrm{squarefree},\ (n,2a)=1\\
                    (\log x)^B\leq|\Kl(a;n)|\leq(\log x)^\beta}}
 |\Kl(a;n)|^\nu
 \leq(\log x)^{\nu\beta}\#\mathcal L_{a,B}(x).
\end{equation}
Consequently,
\begin{equation}\label{eq:lower-LaB}
 \#\mathcal L_{a,B}(x)
 \gg\frac{x}{(\log x)^{1+\nu\beta-\sn(\nu)}}
 \geq\frac{x}{(\log x)^{I(B)+\varepsilon}}.
\end{equation}
Combining \eqref{eq:upper-LaB} and \eqref{eq:lower-LaB}, we obtain
\[
 (\log x)^{-I(B)-\varepsilon}
 \ll\mathbb P_x\bigl(\mathcal L_{a,B}(x)\bigr)
 \ll(\log x)^{-I(B)}.
\]
We take logarithms and divide by \(\log\log x\), then let
\(x\to\infty\) and \(\varepsilon\to0\), in this order. This proves
the corollary.
\end{proof}

\begin{proof}[Proof of Corollary~\ref{cor:Kl-large-value-concentration}]
Fix \(0<\delta<\sn(\nu_B)\). For each choice of sign, let $ z_\pm=1\pm \delta/\sn(\nu_B)$
and
\[
 E_\pm:=\left\{n\in\mathcal L_{a,B}(x):
 \pm\bigl(\omega(n)-\sn(\nu_B)\log\log x\bigr)
 \geq\delta\log\log x\right\}.
\]
Since \(z_+>1\) and \(0<z_-<1\), for either sign and every
\(n\in E_\pm\) we have
\[
 z_\pm^{\omega(n)}|\Kl(a;n)|^{\nu_B}
 \geq(\log x)^{B\nu_B+(\sn(\nu_B)\pm\delta)\log z_\pm}.
\]
The upper bound in Corollary~\ref{cor:moment-order}, together with
\(I(B)=1+B\nu_B-\sn(\nu_B)\), therefore gives
\[
 \mathbb P_x(E_\pm)
 \ll(\log x)^{-I(B)-(\sn(\nu_B)\pm\delta)\log z_\pm\mp\delta}.
 \]
The elementary inequality
\[
 (1+u)\log(1+u)-u\geq\frac{u^2}{4}
 \quad(0<|u|<1)
\]
shows that
\[
 (\sn(\nu_B)\pm\delta)\log z_\pm\mp\delta
 \geq\frac{\delta^2}{4\sn(\nu_B)}.
\]
Hence
\[
 \mathbb P_x(E_+\cup E_-)
 \ll(\log x)^{-I(B)-\frac{\delta^2}{4\sn(\nu_B)}}.
\]
On the other hand, \eqref{eq:lower-LaB} with
\(\varepsilon=\delta^2/(8\sn(\nu_B))\) gives
\[
 \mathbb P_x\bigl(\mathcal L_{a,B}(x)\bigr)
 \gg(\log x)^{-I(B)-\frac{\delta^2}{8\sn(\nu_B)}}.
\]
Dividing these estimates, we obtain
\[
 \mathbb P_x\!\left(
 \left.\left|\frac{\omega(n)}{\log\log x}-\sn(\nu_B)\right|
 \geq\delta\ \right|\ \mathcal L_{a,B}(x)\right)
 \ll_{a,B,\delta}
 (\log x)^{-\frac{\delta^2}{8\sn(\nu_B)}}.
\]
\end{proof}

\begin{proof}[Proof of Corollary~\ref{cor:Kl-omega}]
For \((n,2ab)=1\), a change of variable in \eqref{eq:S-def} shows that
\[
 S(a,b;n)=\sqrt n\,\Kl(ab;n).
\]
The corollary follows from \eqref{eq:large-value-mass}.
\end{proof}

\begin{proof}[Proof of Corollary~\ref{cor:Hecke-Kl-mixed}]
Put $F(n)=|\lambda_f(n)|^\mu$. The Hecke relations and Deligne's bound \cite{Deligne-1974} show that
\(F\) belongs to \(\mathcal M_1(A,B,\varepsilon_0)\), for suitable \(A,B\) and every
fixed \(\varepsilon_0>0\). Moreover, \(G_F=F\), and hence $ G_F(p)=|\lambda_f(p)|^\mu$.

We apply Corollary~\ref{cor:Henriot-cor2} with \(Q(X)=X\). Here
\(D=1\) and \(\rho_Q(p)=1\). Mertens' theorem and
\eqref{eq:lambda-prime-upper} yield
\begin{equation}\label{eq:cor33-45}
	\sum_{x<n\leq x+y}
	|\lambda_f(n)|^\mu|\Kl(a;n)|^\nu\ll
	\frac{y}{\log x}
	\prod_{p\leq x}
	\left(
	1+\frac{\sn(\nu)|\lambda_f(p)|^\mu}{p}
	+O_{\mu,\nu}\left(p^{-\frac98}\right)
	\right).
\end{equation}
It follows from \eqref{eq:Hecke-ST-prime} that
\[
 \log\prod_{p\leq x}
 \left(
 1+\frac{\sn(\nu)|\lambda_f(p)|^\mu}{p}
 +O_{\mu,\nu}\left(p^{-\frac98}\right)
 \right)
 \leq
 \bigl(\sn(\mu)\sn(\nu)+o(1)\bigr)\log\log x.
\]
On inserting this into \eqref{eq:cor33-45}, Corollary~\ref{cor:Hecke-Kl-mixed} follows.
\end{proof}

\begin{proof}[Proof of Corollary~\ref{cor:Hecke-Kl-noncoincidence}]
	Fix \(\varepsilon>0\). Since \(\sn(0)=1\) and \(\sn\) is
	continuous at \(0\), we may fix \(\nu>0\) such that $1+\sn(2\nu)-2\sn(\nu)<\varepsilon$.
	Denote
	\(
	\mathcal S(x)
	=\{n:x/2<n\leq x,\  \mu^2(n)=1,\ (n,2a)=1\}\) and 
	\[
	\mathcal D(x)
	=\{n\in\mathcal S(x):\lambda_f(n)\ne\pm\Kl(a;n)\}.
	\]
	For $n\in\mathcal S(x)\setminus\mathcal D(x)$, we have $|\lambda_f(n)|=|\Kl(a;n)|$. Thus, Corollary~\ref{cor:Hecke-Kl-mixed} yields
	\begin{equation}\label{eq:fourier-kl-nu}
		\sum_{n\in\mathcal S(x)\setminus\mathcal D(x)}
		|\Kl(a;n)|^\nu=\sum_{n\in\mathcal S(x)\setminus\mathcal D(x)}
		|\lambda_f(n)|^{\frac{\nu}{2}}|\Kl(a;n)|^{\frac{\nu}{2}}
		\ll x(\log x)^{\sn(\nu/2)^2-1+\varepsilon'},
	\end{equation}
	where the constant $\varepsilon'>0$ is sufficiently small. On the other hand, Theorem~\ref{thm:lower} yields
	\begin{equation}\label{eq:fourier-kl-nu-2}
	\sum_{n\in\mathcal S(x)}|\Kl(a;n)|^\nu
	\gg x(\log x)^{\sn(\nu)-1}.
	\end{equation}
	By the Cauchy--Schwarz inequality, we get from \eqref{eq:sato-moment} that
	\(\sn(\nu/2)^2<\sn(\nu)\). Subtracting \eqref{eq:fourier-kl-nu} into \eqref{eq:fourier-kl-nu-2}, we obtain
   \begin{equation}\label{eq:Dset-lower}
	\sum_{n\in\mathcal D(x)}|\Kl(a;n)|^\nu
	\gg x(\log x)^{\sn(\nu)-1}.
	\end{equation}
	We use the Cauchy--Schwarz inequality again and the upper bound in
	Corollary~\ref{cor:moment-order}, and get
	\[
	\sum_{n\in\mathcal D(x)}|\Kl(a;n)|^\nu
	\leq \big(\#\mathcal D(x)\big)^{\frac{1}{2}}
	\Big(\sum_{n\leq x}|\Kl(a;n)|^{2\nu}\Big)^{\frac{1}{2}}
	\ll \big(\#\mathcal D(x)\big)^{\frac{1}{2}}
	\big(x(\log x)^{\sn(2\nu)-1}\big)^{\frac{1}{2}}.
	\]
Combining this with \eqref{eq:Dset-lower}, we derive
	\[
	\#\mathcal D(x)
	\gg
	\frac{x}{(\log x)^{1+\sn(2\nu)-2\sn(\nu)}}
	\gg
	\frac{x}{(\log x)^\varepsilon}.
	\]
	This completes the proof of Corollary~\ref{cor:Hecke-Kl-noncoincidence}.
\end{proof}

\begin{proof}[Proof of Theorem~\ref{thm:LS}]
Selberg's identity \cite[(4.10)]{Iwaniec-1997} yields
\begin{equation}\label{eq:Selberg-Id}
	S(u,u;n)=\sum_{d\mid (|u|,n)} \sqrt{nd}\Kl\bigg(\frac{u^2}{d^2};\frac{n}{d}\bigg).
\end{equation}
By the Cauchy--Schwarz inequality, we thus get
\begin{equation}\label{eq:LS-Cauchy}
 \sum_{n\leq x}\frac{|S(u,u;n)|^2}{n^2}
 \leq
 \tau(|u|)
 \sum_{\substack{d\mid |u|\\d\leq x}}
 \sum_{q\leq x/d}
 \frac{\left|\Kl\!\left(\frac{u^2}{d^2};q\right)\right|^2}{q}.
\end{equation}
The upper bound in Corollary~\ref{cor:moment-order}, with
\(z=1\) and \(\nu=2\), asserts that
\[
\sum_{n\leq x}|\Kl(a;n)|^2\ll x
\]
uniformly for \(a\ne0\). This estimate, together with partial
summation in \eqref{eq:LS-Cauchy}, proves the required upper bound.

For the lower bound, suppose that \(0<|u|\leq x^A\). Theorem~\ref{thm:lower} with \(z=1\) and \(\nu=2\) yields
\begin{equation}\label{eq:Kl-sec-lower}
	 \sum_{\substack{n\leq y\\(n,u)=1}}|\Kl(u^2;n)|^2
	\gg y\prod_{p\mid u}\left(1+\frac{1}{p}\right)^{-1}
\end{equation}
for any $\sqrt{x}\leq y\leq x$, where the implied constant depends only on \(A\).
For \((n,u)=1\), Selberg's identity \eqref{eq:Selberg-Id} reduces to
\(S(u,u;n)=\sqrt n\,\Kl(u^2;n)\). Thus, by partial summation, it follows from \eqref{eq:Kl-sec-lower} that
\[
\begin{aligned}
 \sum_{n\leq x}\frac{|S(u,u;n)|^2}{n^2}
 &\geq\sum_{\substack{n\leq x\\(n,u)=1}}
 \frac{|\Kl(u^2;n)|^2}{n}\\
 &\geq\int_{\sqrt x}^{x}\frac{1}{y^2}
 \sum_{\substack{n\leq y\\(n,u)=1}}|\Kl(u^2;n)|^2\,dy\\
 &\gg(\log x)\prod_{p\mid u}\left(1+\frac{1}{p}\right)^{-1}.
\end{aligned}
\]
\end{proof}

\begin{proof}[Proof of Corollary~\ref{cor:spectral-variance}]
The number variance of Li and Sarnak
\cite[Theorem~1.4]{LiSarnak-2004} admits the asymptotic formula
\begin{equation*}
	\Sigma_h(T,L)=R(T,L)+O\!\left(\frac{T}{L^2}\right)
\end{equation*}
uniformly in the range \eqref{eq:spectral-window}, where 
\[
R(T,L)=\frac{T}{\pi^6L^2}
\left(\int_0^\infty\xi^5\psi(\xi)\,d\xi\right)
\sum_{\substack{u,n\geq1\\(u,n)=1}}
\prod_{p\mid n}(1-p^{-2})^{-2}
\frac{|S(u,u;n)|^2}{u^2n^2}
\left|\widehat h\!\left(\frac{\log(Tn/u)}{\pi L}\right)\right|^2.
\]

Since \(\operatorname{supp}\widehat h\subset(-1,1)\), a nonzero
term in \(R(T,L)\) must satisfy $n<u e^{\pi L}/T \leq uT^{\frac{1}{121}}$.
Moreover, $\widehat h(\xi)\ll 1$ and $1\leq\prod_{p\mid n}(1-p^{-2})^{-2}\leq\zeta(2)^2$. Note that  \(L\asymp\log T\). Thus, the upper bound in Theorem~\ref{thm:LS} yields
\begin{equation}\label{eq:R-upperbound}
	\begin{aligned}
		R(T,L)
		&\ll\frac{T}{L^2}
		\sum_{u\geq1}\frac{1}{u^2}
		\sum_{n\leq uT^{\frac{1}{121}}}
		\frac{|S(u,u;n)|^2}{n^2}\\
		&\ll\frac{T}{L^2}
		\sum_{u\geq1}\frac{\tau(u)^2\log(uT^{\frac{1}{121}})}{u^2}
		\ll\frac{T}{L}.
	\end{aligned}
\end{equation}

For the lower bound, it is clear that
\[
R(T,L)\gg \frac{T}{L^2}
\sum_{n\geq1}
\frac{|S(1,1;n)|^2}{n^2}
\left|\widehat h\!\left(\frac{\log(Tn)}{\pi L}\right)\right|^2.
\]
By the choice of $h$, there exist $b_1,b_2$ such that \((1+\delta)^{-1}<b_1<b_2<1\) and \(|\widehat h|\gg 1\) on \([b_1,b_2]\). By partial summation and the estimate \eqref{eq:cor41-fixeda-2}, we get
\[
\begin{aligned}
	R(T,L)
	&\gg \frac{T}{L^2}
	\sum_{e^{\pi Lb_1}/T<n\leq e^{\pi Lb_2}/T}
	\frac{|S(1,1;n)|^2}{n^2} \\
	&\gg \frac{T}{L^2}
	\sum_{e^{\pi Lb_1}/T<n\leq e^{\pi Lb_2}/T}
	\frac{|\Kl(1;n)|^2}{n} \\
	&\gg \frac{T}{L}.
\end{aligned}
\]
Combining this with \eqref{eq:R-upperbound}, we conclude
\eqref{eq:spectral-variance-order}.

To prove \eqref{eq:spectral-mean-square}, we apply \eqref{eq:spectral-variance-order} with a fixed nonnegative smooth function \(\psi\) supported in \([4/3,5/3]\), with \(\int_0^\infty\psi(x)\,dx=1\), and take
\(L=(1+\delta)\log T/\pi\). By the definition of \(N(t)\) and the evenness of \(h\), we have
\[
 N_h(t,L)
 =\int_0^\infty h(L(t-\xi))\,dN(\xi)
 =\int_0^\infty h(L(\xi-t))\,dN(\xi).
\]
By \eqref{eq:spectral-Weyl} and integration by parts, we obtain
\[
 N_h(t,L)-\frac{t}{6L}
 =-\int_0^\infty Lh'(L(\xi-t))S(\xi)\,d\xi+O(\log (1+|t|)),
\]
where we used the fact \(\int_{\R}h(x)dx=1\).  See also \cite[(6.11)]{LiSarnak-2004}. For any \(t\in[4T/3,5T/3]\), it follows from the Cauchy--Schwarz inequality that
\[
 \left|N_h(t,L)-\frac{t}{6L}\right|^2
 \ll(\log T)^2+
 \int_0^\infty |Lh'(L(\xi-t))|S(\xi)^2\,d\xi.
\]
By the choice of \(\psi\), we get
\[
 T\Sigma_h(T,L)
 \ll T(\log T)^2+\int_0^\infty S(\xi)^2
 \bigg(\int_{\frac{4T}{3}}^{\frac{5T}{3}}
 |Lh'(L(\xi-t))|\,dt\bigg)d\xi.
\]
Since \(h'\) is integrable, the inner integral is at most
\(\|h'\|_1\), uniformly in \(\xi\).
For \(\xi \notin [T,2T]\), we have $ |\xi-t|\gg T+\xi$. Thus, the rapid decay of \(h'\) yields 
\[
 \int_{\frac{4T}{3}}^{\frac{5T}{3}}|Lh'(L(\xi-t))|\,dt
 \ll\frac{T}{(T+\xi)^5}
 \leq\frac{1}{(T+\xi)^4}.
\]
By \cite[(1.4)]{LiSarnak-2004}, \(S(\xi)\ll 1+\xi\). Hence the contribution
from \(\xi\notin[T,2T]\) is
\[
 \ll\int_0^\infty\frac{(1+\xi)^2}{(T+\xi)^4}\,d\xi
 \ll\frac{1}{T}.
\]
It follows that
\begin{equation}\label{eq:spectral-transfer}
 T\Sigma_h(T,L)
 \ll T(\log T)^2+\int_T^{2T}S(\xi)^2\,d\xi.
\end{equation}
The lower bound in \eqref{eq:spectral-variance-order}
shows that \(T\Sigma_h(T,L)\gg T^2/\log T\). The estimate \eqref{eq:spectral-mean-square} now follows from \eqref{eq:spectral-transfer}. The \(\Omega\)-estimate follows from
\[
 \sup_{T\leq t\leq2T}|S(t)|^2
 \geq\frac{1}{T}\int_T^{2T}S(t)^2\,dt
 \gg\frac{T}{\log T}.
\]
\end{proof}

\begin{proof}[Proof of Corollary~\ref{cor:ASS}]
Since the shifts are distinct, the polynomials
\(Q_j(X)=X+h_j\) are irreducible and pairwise coprime. Corollary~\ref{cor:normalized}
with \(z_j=2^{-\nu_j}\) therefore yields
\[
 \sum_{x<n\leq2x}
 \prod_{j=1}^s
 \left(\frac{|\Kl(a;n+h_j)|}{2^{\omega(n+h_j)}}\right)^{\nu_j}
 \ll x(\log x)^{\sum_{j=1}^s2^{-\nu_j}\sn(\nu_j)-s},
\]
where the implied constant depends only on \(\boldsymbol\nu,\mathbf h\). The dyadic argument completes the proof.
\end{proof}

\vskip 2mm
\noindent\textbf{Acknowledgments.}
The authors thanks Ping Xi for drawing the authors' attention to the work of Fouvry and Michel during the Working Seminar on Analytic Number Theory at TSIMF in February 2024.

\vskip 2mm

\noindent\textbf{AI use.}
We used ChatGPT to assist in exploring the lower-bound argument in Theorem~\ref{thm:lower} and to improve the exposition in parts of the manuscript. We take full responsibility for the mathematical arguments, results, and final text.

\end{document}